\documentclass[12pt]{amsart}
\usepackage{amsmath,amsthm,amsfonts,latexsym,amssymb,amscd,color,tikz}
\usepackage{chemarr}
\usepackage{graphicx}
\newtheorem{thm}{Theorem}[section]
\newtheorem{cor}[thm]{Corollary}
\newtheorem{lm}[thm]{Lemma}

\newtheorem{clm}[thm]{Claim}

\newtheorem*{clm*}{Claim}
\theoremstyle{definition}

\newtheorem{remrk}[thm]{Remark}

\numberwithin{equation}{section}

\newcommand{\al}[1]{\mathbf{#1}}

\DeclareMathOperator{\Inj}{Inj}
\DeclareMathOperator{\id}{\textrm{id}}     

\newcommand{\wec}[1]{{\mathbf{#1}}}  

\newcommand{\clo}[1]{\mathsf{#1}}
\DeclareMathOperator{\dom}{dom}
\DeclareMathOperator{\supp}{supp}

\DeclareMathOperator{\Sym}{Sym}
\DeclareMathOperator{\Alt}{Alt}

\newcommand{\LC}{\Lambda}
\newcommand{\ULC}{\Upsilon}
\newcommand{\PP}{\mathcal{P}}

\newcommand{\End}{\textrm{End}}

\newcommand{\CCf}{\mathcal{C}}

\let\phi=\varphi
\let\epsilon=\varepsilon
\let\bar=\overline
\let\hat=\widehat
\let\tilde=\widetilde

\def\upsupseteq{\rotatebox[origin=c]{90}{$\subseteq$}}
\def\upeq{\rotatebox[origin=c]{-90}{$=$}}

\newcommand{\qedsymbdiamond}{\renewcommand{\qedsymbol}{$\diamond$}}

\begin{document}

\title[Ultralocally Closed Clones]%
      {Ultralocally Closed Clones}

\author{Keith A. Kearnes}
\address[Keith A. Kearnes]{Department of Mathematics\\
University of Colorado\\
Boulder, CO 80309-0395\\
USA}
\email{Kearnes@Colorado.EDU}

\author{\'Agnes Szendrei}
\address[\'Agnes Szendrei]{Department of Mathematics\\
University of Colorado\\
Boulder, CO 80309-0395\\
USA}
\email{Szendrei@Colorado.EDU}

\thanks{This material is based upon work supported by
the National Science Foundation grant no.\ DMS 1500254,
the Hungarian National Foundation for Scientific Research (OTKA)
grant no.\ K115518,
and the National Research, Development and Innovation Fund of
Hungary (NKFI) grant no.\ K128042.}


\begin{abstract}
  Given a clone $\clo{C}$ on a set $A$, we characterize the
  clone of operations on $A$ which are local term operations
  of every ultrapower
  of the algebra $\langle A;\clo{C}\rangle$.
\end{abstract}

\maketitle

\section{Introduction}
The Baker--Pixley Theorem asserts that if a clone $\clo{C}$ on a finite
set $A$ contains a $d$-ary near unanimity operation ($d\ge 3$),
then every operation that
preserves all compatible relations of arity $<d$ of
the algebra $\langle A;\clo{C}\rangle$ 
belongs to $\clo{C}$. This theorem does not
extend in unmodified form to clones on infinite sets.
Rather, the result 
is that
if a clone $\clo{C}$ on an infinite
set $A$ contains a $d$-ary near unanimity operation,
then every operation that
preserves all compatible relations of arity $<d$ of
the algebra $\langle A;\clo{C}\rangle$ 
belongs to the \emph{local closure} of $\clo{C}$.

``Local closure'' is a closure operator on the lattice of clones
on $A$. We denote the local closure of a clone $\clo{C}$
by $\LC_{\omega}(\clo{C})$, where we use capital Lambda
to stand for ``local''. This closure operator is useful
for translating results about clones on finite sets
to locally closed clones on arbitrary sets.

The drawbacks of passing
from a clone to its local closure are that (i) there
are relatively few locally closed clones on any
infinite set, and (ii) the local closure
of a clone is a coarse approximation
to the clone. Regarding (i),
every clone on a finite set is locally closed, but
on an infinite set of cardinality $\nu$ there are
$2^{2^{\nu}}$-many clones, and only $2^{\nu}$-many
are locally closed (see, e.g., \cite[p.~396]{goldstern-pinsker}).
Regarding (ii), the local closure of a simple $R$-module
always agrees with the $\textrm{End}(V)$-module structure
on a vector space $V$. This may be regarded
as a `coarse' approximation
to the $R$-module structure since, for example,
the ring $\textrm{End}(V)$ typically has many nontrivial
idempotents while $R$ need not have any.

In this paper, we introduce a collection
of finer closure operators on clone
lattices, the most interesting of which is
called ``ultralocal closure''. We denote the ultralocal closure
of a clone $\clo{C}$ 
by $\ULC_{\omega}(\clo{C})$, with capital Upsilon
to stand for ``ultralocal''.
The concept of  ultralocal closure is inspired by the work
of Vaggione in \cite{vaggione}.
We shall find that
\begin{itemize}
\item $\clo{C}\subseteq \ULC_{\omega}(\clo{C}) \subseteq \LC_{\omega}(\clo{C})$
  (the ultralocal closure of $\clo{C}$
  is contained in the local closure of $\clo{C}$), 
\item the number of ultralocally closed clones on an infinite
  set of cardinality $\nu$ is large (= $2^{2^{\nu}}$), and
\item $\ULC_{\omega}(\clo{C})$ can replace the use of
  $\LC_{\omega}(\clo{C})$ in some arguments that extend results
  about clones on finite sets to clones on infinite sets
  (e.g., the Baker--Pixley Theorem).
\end{itemize}  

In fact, our work here covers a little more than we have
described so far. Namely, for every set $A$
and every cardinal $\kappa$
we shall define the \emph{$\kappa$-ultraclosure} of a clone $\clo{C}$
on $A$, written $\ULC_{\kappa}(\clo{C})$.
We say a clone is \emph{$\kappa$-ultraclosed} if
$\ULC_{\kappa}(\clo{C}) = \clo{C}$.
It will follow from the definitions
that $\ULC_1(\clo{C})$ is the clone of all operations on $A$ and
\[
\ULC_1(\clo{C})\supseteq \ULC_2(\clo{C})\supseteq \dots
\supseteq \ULC_\omega(\clo{C})\supseteq \ULC_{\omega_1}(\clo{C})\supseteq
\dots \supseteq \clo{C}.
\]
Then, our main results are:
\begin{enumerate}
\item
  A characterization of
  the $\omega$-ultraclosure (i.e., the ultralocal closure)
  of a clone, $\ULC_{\omega}(\clo{C})$
  (Theorem~\ref{thm-char-lambda} and
  Corollary~\ref{cor-char-kappa}).
\item
  A proof, using the characterization theorem,
  of a version of the Baker--Pixley Theorem:
  every clone containing a $d$-ary near unanimity operation ($d\ge3$)
  satisfies $\ULC_d(\clo{C}) = \clo{C}$ (Theorem~\ref{thm-nu}).
  (The original proof of this statement, using different
  arguments and terminology, is due to Vaggione in
  \cite{vaggione}.)
\item
  A proof, using the characterization theorem, that the clone of any simple
  module is ultralocally closed
  (Theorem~\ref{thm-simple-module}).
\item
  We exhibit examples of clones that are, or are not, ultralocally closed
  (Section~\ref{Section6}).
\end{enumerate}

\section{Preliminaries}
\label{Section2}
Throughout this paper, $A$ and $I$ will denote nonempty sets.
By a \emph{clone} we will mean a clone of operations on some set $A$,
that is, a set of finitary operations on $A$ that contains the projection
operations and is closed under superposition.
The largest clone on $A$ is the clone $\clo{O}_A$ of all operations on $A$.

Fix $A$ and $I$.
For any ultrafilter $\mathcal{U}$ on $I$, the ultrapower $A^I/\mathcal{U}$
of $A$ consists of the equivalence classes
$\wec{a}/\mathcal{U}$ ($\wec{a}=(a_i)_{i\in I}\in A^I$)
of the equivalence relation $\equiv_{\mathcal{U}}$ on $A^I$ defined by
\[
(a_i)_{i\in I}\equiv_{\mathcal{U}}(b_i)_{i\in I}
\qquad\text{if and only if}\qquad
\{i\in I:a_i=b_i\}\in\mathcal{U}.
\]
The diagonal map
$\delta\colon A\to A^I/\mathcal{U}$, $a\mapsto(a)_{i\in I}/\mathcal{U}$
is injective,
therefore $A^I/\mathcal{U}$ may be
viewed as an
extension of $A$, via $\delta$.

For every $n$-ary operation $f\colon A^n\to A$ on $A$, and for every
ultrafilter $\mathcal{U}$ on some set $I$, $f$ has an extension
$f_{\mathcal{U}}$ to the ultrapower $A^I/{\mathcal U}$ of $A$,
defined as follows:
\[
f_{\mathcal{U}}(\wec{a}_1/{\mathcal U},\ldots,\wec{a}_n/{\mathcal U})=
f(\wec{a}_1,\ldots,\wec{a}_n)/{\mathcal U}
\quad\text{for all $\wec{a_1},\dots,\wec{a}_n\in A^I$},
\]
where $f$ on the right hand side acts coordinatewise on elements of $A^I$.
For any clone $\clo{C}$ on $A$ and ultrafilter $\mathcal{U}$ in $I$, we get
a clone $\clo{C}_{\mathcal{U}}$ on $A^I/\mathcal{U}$ by defining 
\[
\clo{C}_\mathcal{U}:=\{t_{\mathcal{U}}:t\in\clo{C}\}.
\]
This is the clone of the ultrapower $\langle A;\clo{C}\rangle^I/\mathcal{U}$
of the algebra $\langle A;\clo{C}\rangle$. The diagonal map
$\delta\colon A\to A^I/\mathcal{U}$ is an elementary embedding
$\langle A;\clo{C}\rangle\to\langle A;\clo{C}\rangle^I/\mathcal{U}
=\langle A^I/\mathcal{U};\clo{C}_{\mathcal{U}}\rangle$, therefore the algebra
$\langle A;\clo{C}\rangle^I/\mathcal{U}
=\langle A^I/\mathcal{U};\clo{C}_{\mathcal{U}}\rangle$
may be viewed as an elementary
extension of $\langle A;\clo{C}\rangle$.

Let $f$ be an $n$-ary operation on $A$
and let
$\clo{C}$ be an arbitrary clone on $A$.
Furthermore, let $\kappa>0$ and $\lambda$ be cardinals.
We say that $f$ is 
\emph{$\lambda$-interpolable by $\clo{C}$},
if whenever $S\subseteq A^n\,\bigl(=\dom(f)\bigr)$
satisfies $|S|\le\lambda$, there is some $n$-ary $t\in\clo{C}$
such that $f|_S=t|_S$. (See Figure~\ref{fig-k-int} for the case when
$\lambda$ is finite.)
We define the \emph{$\kappa$-closure}, $\LC_{\kappa}(\clo{C})$,
of $\clo{C}$ to consist of all operations on $A$
that are $\lambda$-interpolable by $\clo{C}$ for every $\lambda<\kappa$.
(Notice the strict $<$ here!)
The clone $\clo{C}$ is called \emph{$\kappa$-closed} if
$\clo{C}=\LC_{\kappa}(\clo{C})$.
In the special case $\kappa=\omega$,
the $\omega$-closure $\LC_{\omega}(\clo{C})$
of $\clo{C}$ is called the \emph{local closure} of $\clo{C}$, and 
$\clo{C}$ is called \emph{locally closed} if $C=\LC_{\omega}(\clo{C})$.

\begin{figure}[h]
\setlength{\unitlength}{1truemm}
\begin{picture}(150,40)
\put(53,0){%
\begin{tikzpicture}[scale=1.2]
\draw[fill] (1.6,1.7) circle (.05);
\node at (1.8,1.7) {{$s_0$}};
\draw[fill] (1.3,1.2) circle (.05);
\node at (1.1,1.2) {{$s_1$}};
\draw[fill] (1.7,.9) circle (.05);
\node at (1.7,.7) {{$s_2$}};
\draw[fill] (2.5,.9) circle (.05);
\node at (2.5,.7) {{$s_{\lambda-1}$}};
\node at (2.12,.9) {{\dots}};
\draw[line width=1.2pt] (0,0) -- (3,0) -- (3,2) -- (0,2) -- cycle;
\node at (3.3,2.1) {$A^n$};
\node at (1.5,-.4) {$f(s_i)=t(s_i)$ for all $s_i\in S$};
\end{tikzpicture}
}
\end{picture}
\caption{$f$ is $\lambda$-interpolable ($\lambda<\omega$)}
\label{fig-k-int}
\end{figure}
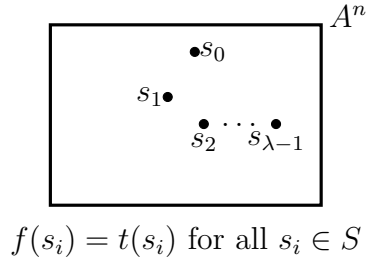

For $f$, $\clo{C}$, and $\kappa$, $\lambda$ as before, we will say that
$f$ is \emph{$\lambda$-ultrainterpolable by $\clo{C}$},
if $f_{\mathcal{U}}$ is $\lambda$-interpolable by $\clo{C}_{\mathcal{U}}$
for every ultrafilter $\mathcal{U}$ on any set $I$.
That is, $f$ is $\lambda$-ultrainterpolable by $\clo{C}$, if
for any ultrafilter $\mathcal{U}$ on any set $I$, we have that
whenever $S\subseteq (A^I/{\mathcal U})^n$
satisfies $|S|\le\lambda$, there is some $n$-ary $t\in\clo{C}$
such that $(f_{\mathcal{U}})|_S=(t_{\mathcal{U}})|_S$.
We define the \emph{$\kappa$-ultraclosure}, $\ULC_{\kappa}(\clo{C})$,
of $\clo{C}$ to consist of all operations on $A$
that are $\lambda$-ultrainterpolable by $\clo{C}$ for every $\lambda<\kappa$.
(Strict $<$ here, too!)
The clone $\clo{C}$ is called \emph{$\kappa$-ultraclosed} if
$\clo{C}=\ULC_{\kappa}(\clo{C})$.
In the special case $\kappa=\omega$,
the $\omega$-ultraclosure $\ULC_{\omega}(\clo{C})$
of $\clo{C}$ is called the \emph{ultralocal closure} of $\clo{C}$, and 
$\clo{C}$ is called \emph{ultralocally closed} if
$\clo{C}=\ULC_{\omega}(\clo{C})$.

If $f$ is $\lambda$-ultrainterpolable by $\clo{C}$, then
$f$ is $\lambda$-interpolable by $\clo{C}$, for the following reason.
Assume that $f$ is $\lambda$-ultrainterpolable by $\clo{C}$, and
that $\mathcal{U}$ is a principal
ultrafilter on some set $I$ with $\{u\}\in\mathcal{U}$ ($u\in I$).
Since $f$ is $\lambda$-ultrainterpolable by $\clo{C}$,
$f_{\mathcal{U}}$ is $\lambda$-interpolable by $\clo{C}_{\mathcal{U}}$.
Since $\mathcal{U}$ is generated by $\{u\}$,
the equivalence relation $\equiv_{\mathcal{U}}$
is the kernel of the projection $A^I\to A$ onto the $u$-th coordinate, so
$\delta\colon A\to A^I/\mathcal{U}$ is a bijection.
Therefore, up to renaming elements of the base sets via
$\delta$, $\clo{C}_{\mathcal{U}}$ and $\clo{C}$ are the same
clone, and $f_{\mathcal{U}}$ and $f$ are the same operation.
Hence, $f$ is $\lambda$-interpolable by~$\clo{C}$.

The argument just given
proves statement (1) of the lemma below. Statement (2) is an immediate
consequence of the definitions. Statement (3) follows from the fact that
for a finite set $A$, the elementary
embedding $\delta\colon A\to A^I/\mathcal{U}$
is an isomorphism for any ultrafilter $\mathcal{U}$ on any set $I$.

\begin{lm}
  \label{lm-incl}
  For arbitrary clone $\clo{C}$ on a set $A$, and
  for any cardinals $\mu,\nu\,(>0)$,
  \begin{enumerate}
  \item[{\rm(1)}]
    $\clo{C}\subseteq \ULC_\mu(\clo{C})\subseteq \LC_\mu(\clo{C})$, and
  \item[{\rm(2)}]
    $\clo{C}\subseteq\ULC_\nu(\clo{C})\subseteq\ULC_\mu(\clo{C})$
    if $\mu\le\nu$.
  \item[{\rm(3)}]
    For finite $A$,
    \begin{itemize}
    \item
      $\ULC_\mu(\clo{C})=\LC_\mu(\clo{C})$, moreover,
    \item
      $\clo{C}=\ULC_\mu(\clo{C})=\LC_\mu(\clo{C})$ if $\mu$ is infinite.
    \end{itemize}
  \end{enumerate}
\end{lm}  

Statement~(3) of the lemma shows that for clones on finite sets
the closure operators $\ULC_\mu$ ($\mu>0$) are not new.
Therefore our results in the forthcoming sections are interesting only
for clones on infinite sets.

Since every operation $f$ on a set $A$ is $0$-interpolable by any clone
$\clo{C}$ on $A$, we have that $\ULC_1(\clo{C})=\LC_1(\clo{C})=\clo{O}_A$.
Hence, statements (1)--(2) of Lemma~\ref{lm-incl} can be summarized as follows:
\begin{align*}
    \clo{O}_A &{}= \LC_1(\clo{C}) \supseteq \,\LC_2(\clo{C}) \supseteq
    \,\LC_3(\clo{C}) \supseteq\ \ \cdots\ \ \supseteq
    \LC_\omega(\clo{C}) \supseteq \,\LC_{\omega_1}(\clo{C}) \supseteq
    \ \ \cdots\ \ \clo{C}\\
    & \kern28pt \upeq \kern40pt \upsupseteq \kern40pt \upsupseteq
    \kern80pt \upsupseteq \kern45pt \upsupseteq \kern60pt \upeq
    \\
    \clo{O}_A &{}= \ULC_1(\clo{C}) \supseteq \ULC_2(\clo{C}) \supseteq
    \ULC_3(\clo{C}) \supseteq\ \ \cdots\ \ \supseteq
    \ULC_\omega(\clo{C}) \supseteq \ULC_{\omega_1}(\clo{C}) \supseteq
    \ \ \cdots\ \,\phantom{.}\clo{C}.
\end{align*}  

For any cardinal $\kappa>0$,
the property that a clone $\clo{C}$ is $\kappa$-closed can be
rephrased in terms of invariant relations, as stated in Lemma~\ref{lm-invrel}
below. For $\kappa=\omega$ the results of this lemma are due to
Romov, \cite{romov}. The statements carry over from $\kappa=\omega$
to arbitrary cardinals
$\kappa>0$ without any essential changes.

For any set $R$ of (finitary or infinitary) relations on a set $A$,
we will use the notation $\clo{Pol}(R)$ for
the clone consisting of all (finitary) operations on $A$
that preserve every relation in $R$.

\begin{lm}{\rm(cf.~\cite{romov})}
  \label{lm-invrel}
Let $\kappa$ be a nonzero cardinal, $\clo{C}$ a clone on a set $A$, and
let $R$ be a set of relations of arity $<\kappa$ on $A$.
\begin{enumerate}
\item[{\rm(1)}]
  $\clo{Pol}(R)$ is a $\kappa$-closed clone on $A$.
\item[{\rm(2)}]
  If $\clo{C}\subseteq\clo{Pol}(R)$ (that is, if $R$ consists of
  invariant relations of $\clo{C}$), then 
  \[
  \clo{C}
  \subseteq\LC_\kappa(\clo{C})\subseteq
  \clo{Pol}(R).
  \]
\item[{\rm(3)}]
  $\LC_\kappa(\clo{C})=\clo{Pol}(R_{\clo{C}})$ for the set
  $R_{\clo{C}}$ of all invariant relations 
  of arity $<\kappa$ of $\clo{C}$.
\end{enumerate}
\end{lm}  

Using Lemma~\ref{lm-incl}(1) one can expand the sequence of
inclusions in (2) to 
  \[
  \clo{C}\subseteq\ULC_\kappa(\clo{C})
  \subseteq\LC_\kappa(\clo{C})\subseteq
  \clo{Pol}(R).
  \]
This will be useful for us, because it shows that if a property of clones 
is expressible by the preservation of some invariant relation,
then this property is passed on
from $\clo{C}$ to $\LC_\kappa(\clo{C})$, and hence to
$\ULC_\kappa(\clo{C})$, for large enough $\kappa$.

Next we discuss special cases of Lemma~\ref{lm-invrel}
when $\clo{C}$ is 
an essentially unary clone, a module clone, or a product
clone.
The effect of $\ULC_\kappa$ on unary clones and product clones will be
employed in Section~\ref{Section6} to construct large families of clones on
infinite sets that are not ultralocally closed, while
the effect of $\ULC_\kappa$ on module clones will be used in
Section~\ref{Section5} where we show that the clone of every simple module is
ultralocally closed.

In our first corollary to Lemma~\ref{lm-invrel}
a clone $\clo{C}$ is called \emph{essentially unary} if every operation in
$\clo{C}$ depends on at most one of its variables.

\begin{cor}
  \label{cor-specprop1}
  Let $\clo{C}$ be a clone and $\kappa$ a nonzero cardinal. 
  \begin{enumerate}
  \item[{\rm(1)}]
    If $\clo{C}$ is essentially unary, then so are $\LC_\kappa(\clo{C})$
    and $\ULC_\kappa(\clo{C})$ for every $\kappa\ge4$.
  \item[{\rm(2)}]
    If all unary operations in $\clo{C}$ are injective, then
    $\LC_\kappa(\clo{C})$ and $\ULC_\kappa(\clo{C})$
    have the same property for every $\kappa\ge3$.
  \end{enumerate}  
\end{cor}  

\begin{proof}
  For (1), we use the following easily proved fact.

  \begin{clm}
    \label{clm-essUnary}
    An operation $f$ on a set $A$
    is essentially unary if and only if $f$ preserves the ternary relation
  $\rho_3:=\{(a,b,c)\in A^3:a=b\ \text{or}\ b=c\}$.
  \end{clm}

  \begin{proof}[Proof of Claim~\ref{clm-essUnary}]
    Let $A$ be an arbitrary set.
    It is proved in \cite[Lemma~1.3.1]{pk}
    that an operation $f$ on $A$ is essentially unary
    if and only if $f$ preserves the $4$-ary relation
    $\pi_4:=\{(a,b,c,d)\in A^4:a=b\ \text{or}\ c=d\}$. In other words,
    $\clo{Pol}(\pi_4)$ is the clone of all essentially unary
    operations.

    To prove that the relation $\pi_4$ here can be replaced by $\rho_3$,
    notice that $\clo{Pol}(\rho_3)$ contains all essentially unary operations;
    therefore it suffices to show that
    $\clo{Pol}(\rho_3)\subseteq\clo{Pol}(\pi_4)$.
    This can be done by exhibiting a primitive positive definition for
    $\pi_4$ in terms of $\rho_3$. (See, e.g., \cite[Chapter~2]{pk} for why
    this is sufficient.)

    We claim that the primitive positive formula
    \begin{multline*}
    \Phi(x_0,x_1,x_2,x_3):\equiv\Psi(x_0,x_1,x_2,x_3)\wedge\Psi(x_1,x_0,x_2,x_3)
    \qquad\text{with}\\
    \Psi(x_0,x_1,x_2,x_3):\equiv
    \exists y\,\bigl(\rho_3(x_0,x_1,y)\wedge\rho_3(y,x_2,x_3)\bigr)
    \end{multline*}
    defines $\pi_4$.
    Indeed, it is easy to verify that the relation defined by
    $\Psi(x_0,x_1,x_2,x_3)$ is
    $\{(a,b,c,d)\in A^4:a=b \text{ or } c=d \text{ or } b=c\}$.
    Hence the relation defined by $\Phi(x_0,x_1,x_2,x_3)$ is
    $\{(a,b,c,d)\in A^4:a=b \text{ or } c=d \text{ or } a=b=c\}=\pi_4$.

    This claim is also proved in \cite[Lemma~5.3.2]{bodirsky}.
    \qedsymbdiamond
  \end{proof}
  
  It follows from Claim~\ref{clm-essUnary} that
  if $\clo{C}$ is an essentially unary
  clone, then $\clo{C}\subseteq\clo{Pol}(\rho_3)$. Hence,
  by applying Lemma~\ref{lm-invrel}(2) with $R=\{\rho_3\}$, we get that
  $\clo{C}\subseteq\ULC_\kappa(\clo{C})\subseteq\LC_\kappa(\clo{C})\subseteq
  \clo{Pol}(\rho_3)$ for $\kappa\ge4$.
  This shows that the clone $\LC_\kappa(\clo{C})$ and its subclone,
  $\ULC_\kappa(\clo{C})$,
  are also essentially unary
  if $\kappa\ge4$. The proof of (1) is complete.

  A unary operation $f\colon A\to A$ is injective exactly when it preserves
  the binary ``not equal'' relation $\{(a,b)\in A^2: a\not= b\}$.
  Now, statement~(2) follows in the same way as
  statement~(1).
\end{proof}

\begin{cor}
  \label{cor-specprop2}
  Let $\clo{C}$ be a clone and $\kappa$ a nonzero cardinal. 
    If $\clo{C}$ is the clone of an $R$-module, for some ring $R$,
    with underlying abelian group $\hat{A}=\langle A;+,-,0\rangle$,
    then so are $\LC_\kappa(\clo{C})$ and $\ULC_\kappa(\clo{C})$
    for every $\kappa\ge 4$. 
\end{cor}  

\begin{proof}
  Let ${}_RA$
    be an $R$-module 
  with underlying abelian group $\hat{A}$, and let $\clo{C}$ be the clone
  of term operations of ${}_RA$. It is known (for example, it follows
  from \cite[Proposition~2.1]{szendreiBOOK}) that
  \begin{itemize}
  \item
    the graph of $+$, that is, the ternary relation
  \[
  \gamma(+):=\{(a,b,a+b):a,b\in A\}
  \]
  is preserved by every operation in $\clo{C}$; moreover,
  \item
    the clone $\clo{Pol}\bigl(\gamma(+)\bigr)$ of all operations
    that preserve $\gamma(+)$ coincides with the clone of the module
    ${}_{\End(\widehat{A})}A$, which is $\hat{A}$ as a module over its
    endomorphism ring $\End(\hat{A})$.
  \end{itemize}  
  Consequently, every subclone $\clo{S}$ of the clone of
  ${}_{\End(\hat{A})}A$ such that $\clo{S}$ contains the clone of $\widehat{A}$,
  is the clone of a module ${}_SA$ with underlying abelian group
  $\hat{A}$
  for some subring $S$ of $\End(\hat{A})$; namely, $S$ is
  the ring of all unary operations in $\clo{S}$.
  By Lemma~\ref{lm-invrel}(2),
  each $\ULC_\kappa(\clo{C})$ ($\kappa\ge4$) is one of these clones,
  therefore each $\ULC_\kappa(\clo{C})$ ($\kappa\ge4$) is the clone of a
  module with underlying abelian group $\hat{A}$, as claimed.
\end{proof}

For arbitrary clones $\clo{P}$ on a set $A$ and
$\clo{Q}$ on a set $B$ their \emph{product}, $\clo{P}\times\clo{Q}$,
is the clone on $A\times B$ defined as follows:
for each $0<n<\omega$, the $n$-ary members are the
\emph{product operations}
$g\times h$ where $g$ is an $n$-ary operation in $\clo{P}$ and
$h$ is an $n$-ary operation in $\clo{Q}$.
The product operation $g\times h$ is defined to act coordinatewise on
$A\times B$; that is,
\[
(g\times h)\bigl((a_1,b_1),\dots,(a_n,b_n)\bigr)
=\bigl(g(a_1,\dots,a_n),h(b_1,\dots,b_n)\bigr)
\quad\text{for all $a_i\in A$, $b_i\in B$}.
\]
A clone on $A\times B$ is called a \emph{product clone} if it has the form
$\clo{P}\times\clo{Q}$ for some clones $\clo{P}$ on $A$ and $\clo{Q}$ on
$B$.

\begin{cor}
  \label{cor-specprop3}
  Let $\clo{C}$ be a clone on a set $A\times B$, and let
  $\kappa$ be a nonzero cardinal. 
  If $\clo{C}$ is a  product clone on $A\times B$, then
    so are $\LC_\kappa(\clo{C})$ and $\ULC_\kappa(\clo{C})$
    for every $\kappa\ge 4$.
\end{cor}  

\begin{proof}
  Let $*$ denote the binary operation on $A\times B$
  defined as follows:
  \[
  (a_1,b_1)*(a_2,b_2)=(a_1,b_2)
  \qquad\text{for all $a_1,a_2\in A$ and $b_1,b_2\in B$.}
  \]
  This operation is known as the binary \emph{diagonal operation}
  or the \emph{rectangular band operation of the product}
  $A\times B$.
  Notice that $*$ is the product operation
  $p_1^A\times p_2^B\in\clo{O}_A\times\clo{O}_B$ where $p_1^A$ is
  binary projection to the first variable on $A$, and $p_2^B$ is
  binary projection to the second variable on $B$.
  We will also use the graph of the operation $*$, which is the
  following ternary relation:
  \[
  \gamma(*):=
  \{(u,v,u*v)\in(A\times B)^3:
  u,v\in A\times B\}.
  \]  
We will need the following facts.

  \begin{clm}
    \label{clm-product}
    Let $A,B$ be arbitrary sets.
    \begin{enumerate}
    \item[{\rm(1)}]
      The following conditions on an $n$-ary operation
      $f$ on $A\times B$ are equivalent:
      \begin{itemize}
      \item
        $f=f_A\times f_B$ for some
        $n$-ary operations $f_A$ on $A$ and  $f_B$ on $B$;
      \item
        $f$ commutes with $*$;
      \item
        $f$ preserves the graph $\gamma(*)$ of the operation $*$.
      \end{itemize}
   \item[{\rm(2)}]
  A clone $\clo{C}$ on $A\times B$ is a product clone if and only if
  \begin{enumerate}
  \item[{\rm (i)}]
    $\clo{C}\subseteq\clo{Pol}(\gamma(*))$, i.e., every operation in
    $\clo{C}$ commutes with $*$, and
  \item[{\rm(ii)}]
    $*$ is a member of $\clo{C}$.
  \end{enumerate}
    \end{enumerate}
\end{clm}

\begin{proof}[Proof of Claim~\ref{clm-product}]
  For (1), let $f$ be an $n$-ary operation on $A\times B$, i.e.,
  $f\colon(A\times B)^n\to A\times B$. We will write an $n$-tuple of pairs
  from $A\times B$ as an $n\times 2$ matrix $[\bar{a}\ \bar{b}]$
  with columns $\bar{a}\in A^n$ and $\bar{b}\in B^n$. The rows are
  the pairs $(a_i,b_i)$ ($i<n$) where $\bar{a}=[a_0\ \dots\ a_{n-1}]^T$
  and $\bar{b}=[b_0\ \dots\ b_{n-1}]^T$. Thus, when $*$ is applied
  coordinatewise (down columns)
  to two $n$-tuples,
  $[\bar{a}\ \bar{b}]$ and $[\bar{a}'\ \bar{b}']$ in $(A\times B)^n$,
  we get
  \begin{equation}
    \label{eq1-product}
    [\bar{a}\ \bar{b}]*[\bar{a}'\ \bar{b}']=[\bar{a}\ \bar{b}'].
  \end{equation}

  Let $\tilde{f}_A$ denote the function
  $\tilde{f}_A\colon(A\times B)^n\to A$ obtained from $f$ by composing it with
  the function $A\times B\to A$, $(a,b)\mapsto a$, and similarly, let
  $\tilde{f}_B\colon(A\times B)^n\to B$ be obtained from $f$ by composing it
  with the function $A\times B\to B$, $(a,b)\mapsto b$. We have
  \begin{equation}
    \label{eq2-product}
  f([\bar{a}\ \bar{b}])
  =\bigl(\tilde{f}_A([\bar{a}\ \bar{b}]),\,\tilde{f}_B([\bar{a}\ \bar{b}])\bigr)
  \quad\text{for all $[\bar{a}\ \bar{b}]\in(A\times B)^n$}.
  \end{equation}

  Now we are ready to prove the equivalence of the three conditions in (1).
  The last two of these conditions are different ways of stating the same
  relationship between $f$ and $*$,
  and they are easily seen to be implied
  by the first condition. Therefore it remains to prove that the first
  condition follows from the second.
  The second condition is the
  statement that
  \begin{equation}
    \label{eq3-product}
  f([\bar{a}\ \bar{b}]*[\bar{a}'\ \bar{b}'])=
  f([\bar{a}\ \bar{b}])*f([\bar{a}'\ \bar{b}'])
  \quad\text{for all
    $[\bar{a}\ \bar{b}],[\bar{a}'\ \bar{b}']\in(A\times B)^n$}.
  \end{equation}
  By applying \eqref{eq1-product} and \eqref{eq2-product} we see that the
  left hand side of the equality in \eqref{eq3-product} is
  \[
  f([\bar{a}\ \bar{b}]*[\bar{a}'\ \bar{b}'])=
  f([\bar{a}\ \bar{b}'])=
  \bigl(\tilde{f}_A([\bar{a}\ \bar{b}']),\tilde{f}_B([\bar{a}\ \bar{b}'])\bigr),
  \]
  while the right hand side is
  \[
  f([\bar{a}\ \bar{b}])*f([\bar{a}'\ \bar{b}'])=
  \bigl(\tilde{f}_A([\bar{a}\ \bar{b}]),\tilde{f}_B([\bar{a}'\ \bar{b}'])\bigr).
  \]
  Thus, \eqref{eq3-product} is equivalent to the condition that
  $\tilde{f}_A$ does not depend on the second column of the input matrix
  $[\bar{a}\ \bar{b}]$, and
  $\tilde{f}_B$ does not depend on the first column of the input matrix
  $[\bar{a}'\ \bar{b}']$.
  That is, there exist $f_A\colon A^n\to A$ and $f_B\colon B^n\to B$ such that
  \[
  f([\bar{a}\ \bar{b}])=\bigl( f_A(\bar{a}), f_B(\bar{b})\bigr)
  \quad\text{for all $\bar{a}\in A^n$ and $\bar{b}\in B^n$},
  \]
  or equivalently, there exist $n$-ary operations $f_A$ on $A$ and $f_B$ on $B$
  such that $f=f_A\times f_B$.
  This finishes the proof of (1).

For the forward implication of statement (2), if $\clo{C}$ is a product
  clone on $A\times B$, then (i) holds by part (1) of this claim and
  (ii) holds by the observation made in the paragraph preceding
  Claim~\ref{clm-product} that $*$ is a product operation where each factor
  is a projection.

For the backward implication of statement (2),   
assume that $\clo{C}$ is a clone on $A\times B$ such that
  conditions (i)--(ii) are met. By statement (1) above, (i) implies that
  every operation $f\in\clo{C}$ is a product operation:
  $f=f_A\times f_B$ for some operations $f_A$ on $A$ and $f_B$ on $B$, of the
  same arity as $f$. Let $\clo{P}:=\{f_A:f\in\clo{C}\}$ and
  $\clo{Q}:=\{f_B:f\in\clo{C}\}$.
  It follows that
  $\clo{P}$ is a clone on $A$, $\clo{Q}$ is a clone on $B$, and $\clo{C}$
  is a subclone of $\clo{P}\times\clo{Q}$.
  We claim that  $\clo{C}=\clo{P}\times\clo{Q}$.
  Let $n\ge1$, and consider arbitrary $n$-ary operations
  $g\in\clo{P}$ and $h\in\clo{Q}$. By the definitions of $\clo{P}$ and
  $\clo{Q}$, there exist $n$-ary operations $g',h'\in\clo{C}$ such that
  $g=g'_A$ and $h=h'_B$; that is, $g'=g\times g'_B$ and $h'=h'_A\times h$.
  By condition (ii) we have $*\in\clo{C}$, therefore
  $g\times h=(g\times g'_B)*(h'_A\times h)=g'*h'\in\clo{C}$.
  This shows that $\clo{C}\supseteq\clo{P}\times\clo{Q}$, which
  completes the proof of (2).
\qedsymbdiamond
\end{proof}
  
It follows from Claim~\ref{clm-product}
that if $\clo{C}$ is a product clone on $A\times B$, then
$*\in\clo{C}\subseteq\clo{Pol}(\gamma(*))$.
Therefore, by applying
Lemma~\ref{lm-invrel}(2) with $R=\{\gamma(*)\}$, we obtain that
$*\in\clo{C}\subseteq\ULC_\kappa(\clo{C})\subseteq\LC_\kappa(\clo{C})\subseteq
\clo{Pol}(\gamma(*))$ for $\kappa\ge4$.
Hence, Claim~\ref{clm-product}(2) yields that $\LC_\kappa(\clo{C})$ and
$\ULC_\kappa(\clo{C})$ are both product clones for $\kappa\ge4$.
\end{proof}

\section{Characterizing ultralocal closure}

Our main goal in this section is to characterize the $\kappa$-ultraclosure of
a clone $\clo{C}$ for each 
nonzero cardinal $\kappa\le\omega$.
The main ingredient is the following characterization of the operations
that are $\lambda$-ultrainterpolable by $\clo{C}$ for some
$\lambda<\omega$.

In what follows, a \emph{cover} of a set $X$
is a set $\mathcal{C}\subseteq\PP(X)$
  of subsets of $X$ such that
  $\bigcup\mathcal{C}=X$.
  So, $\mathcal{C}$ is a \emph{finite cover} of $X$ if
$\mathcal{C}$ is a finite set and a cover of $X$.

\begin{thm}
  \label{thm-char-lambda}
  Let $\clo{C}$ be a clone on a set $A$, and let
  $f\colon A^n\to A$ be an $n$-ary operation on $A$ $(0<n<\omega)$.
  The following conditions are equivalent for any
  $\lambda<\omega$.  
  \begin{enumerate}
  \item[$(\dagger)_\lambda$]
    $f$ is $\lambda$-ultrainterpolable by $\clo{C}$.
  \item[$(\ddagger)_\lambda$]
    $A^n\,\bigl(=\dom(f)\bigr)$ has a finite cover
    $\CCf_\lambda\,\bigl(\subseteq \PP(A^n)\bigr)$ such that
    whenever $\mathcal{B}\subseteq\CCf_\lambda$
    satisfies $|\mathcal{B}|\le\lambda$, there exists an $n$-ary
    $t^{[\mathcal{B}]}\in\clo{C}$ 
    such that $f|_{\bigcup{\mathcal{B}}}=t^{[\mathcal{B}]}|_{\bigcup{\mathcal{B}}}$.
  \end{enumerate}    
\end{thm}

Note that
condition $(\ddagger)_\lambda$ and the finite cover $\CCf_\lambda$ involved
both depend on the choice of $f$, $\clo{C}$, and $\lambda$. Dependence on
$f$ and $\clo{C}$ is suppressed in the notation, but when we apply
condition $(\ddagger)_\lambda$, the choice of $f$ and $\clo{C}$ will be
clear from the context.

Condition $(\ddagger)_\lambda$ is illustrated by Figure~\ref{fig-k-uint}.
The figure indicates that $A^n$ has a finite cover
$\CCf_\lambda$
such that for any set
\[
\mathcal{B}=\{B_0,B_1,\dots,B_{\lambda-1}\}\subseteq \CCf_\lambda
\]
consisting of at most $\lambda$ members of $\CCf_\lambda$
there is
a $t^{[\mathcal{B}]}\in\clo{C}$ such that
$f$ and $t^{[\mathcal{B}]}$ agree on
$\bigcup\mathcal{B}$.
Figure~\ref{fig-k-uint} resembles Figure~\ref{fig-k-int}, except that
now we are interpolating $f$ over a set $\mathcal{B}$ of
$\lambda$ regions selected from the finite cover $\CCf_\lambda$
of its domain,
where previously we were interpolating $f$ over a set $S$ of $\lambda$
single points of its domain.

\begin{figure}[h]
\setlength{\unitlength}{1truemm}
\begin{picture}(150,40)
\put(53,0){%
\begin{tikzpicture}[scale=1.2]
\draw[fill=gray] (1.5,2) -- (1.5,1.5) -- (.5,1.5) --
(.5,.8) -- (1.3,.8) -- (1.3,.3) -- (2.3,.3) -- (2.3,1) -- (1.8,1) --
(1.8,1.3) -- (2.2,1.3) -- (2.5,2) -- cycle;
\draw[fill=gray] (2.6,.7) -- (3,.7) -- (3,1.1) -- (2.6,1.1) -- cycle;
\node at (1.27,1.8) {{$B_0$}};
\node at (.3,1.6) {{$B_1$}};
\node at (1.09,.57) {{$B_2$}};
\node at (2.95,1.3) {{$B_{\lambda-1}$}};
\node at (2.7,.5) {{$\dots$}};
\draw[line width=1.2pt] (0,0) -- (3,0) -- (3,2) -- (0,2) -- cycle;
\node at (3.3,2.1) {$A^n$};
\draw[line width=.4pt] (0,0) -- (.6,0) -- (.6,.5) -- (0,.5) -- cycle;
\draw[line width=.4pt] (0,.4) -- (.3,.4) -- (.3,1.2) -- (0,1.2) -- cycle;
\draw[line width=.4pt] (0,1) -- (.5,1) -- (.8,2) -- (0,2) -- cycle;
\draw (.5,.7) circle (.4);
\draw[line width=.4pt] (.5,0) -- (1.2,0) -- (1.4,.9) -- (.7,1) -- cycle;
\draw[line width=.4pt] (.5,.8) -- (1.8,.8) -- (1.8,1.5) -- (.5,1.5) -- cycle;
\draw[line width=.4pt] (.6,1.4) -- (1.6,1.4) -- (1.6,2) -- (.6,2) -- cycle;
\draw[line width=.4pt] (1.5,1.3) -- (2.2,1.3) -- (2.5,2) -- (1.5,2) -- cycle;
\draw[line width=.4pt] (2.5,1) -- (3,1) -- (3,2) -- (2.2,2) -- cycle;
\draw (2.1,1.1) circle (.6);
\draw[line width=.4pt] (1.2,0) -- (2.1,0) -- (2.1,.4) -- (1.2,.4) -- cycle;
\draw[line width=.4pt] (1.3,.3) -- (2.3,.3) -- (2.3,1) -- (1.3,1) -- cycle;
\draw[line width=.4pt] (1.9,0) -- (3,0) -- (3,.8) -- (2.2,.8) -- cycle;
\draw[line width=.4pt] (2.6,.7) -- (3,.7) -- (3,1.1) -- (2.6,1.1) -- cycle;
\node at (1.4,-.4) {$f|_{B_i}=t^{[\mathcal{B}]}|_{B_i}$ for all $B_i\in\mathcal{B}$};
\end{tikzpicture}
  }
\end{picture}
\caption{$f$ is $\lambda$-ultrainterpolable ($\lambda<\omega$)}
\label{fig-k-uint}
\end{figure}
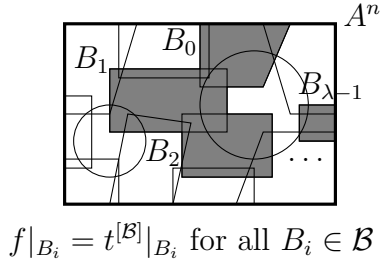

\begin{cor}
  \label{cor-char-kappa}
  Let $\clo{C}$ be a clone on a set $A$, and
  let $\kappa\le\omega$ be a nonzero cardinal.
  The $\kappa$-ultraclosure, $\ULC_\kappa(\clo{C})$, of $\clo{C}$
  consists of all operations $f\colon A^n\to A$ $(0<n<\omega)$
  which satisfy condition $(\ddagger)_\lambda$ from
  Theorem~\ref{thm-char-lambda} for every $\lambda<\kappa$.
\end{cor}  

We will focus primarily on the case $\kappa=\omega$, that is, on
the $\omega$-ultraclosure of clones $\clo{C}$ on infinite sets,
which we also call the \emph{ultralocal closure} of $\clo{C}$. 
In Section~\ref{Section6} we will give examples to show that
there exist clones on infinite sets
that are \emph{not} ultralocally closed (see Theorems~\ref{thm-lots} and
\ref{thm-alt}).

The rest of this section is devoted to the proof of
Theorem~\ref{thm-char-lambda}. 
We start by introducing some terminology and notation, that will allow us
to restate condition $(\ddagger)_\lambda$
of Theorem~\ref{thm-char-lambda}
in a form that is more convenient
for our proof.

Let $\clo{C}$, $f$ with $\dom(f)=A^n$, and $\lambda$ be as
in Theorem~\ref{thm-char-lambda}.
It will be convenient to think of the elements of $A^n$ as columns of length
$n$, and the elements of the set $(A^n)^\lambda$ as
$\lambda$-sequences of column vectors
in
$A^n$, or equivalently, 
as $n\times \lambda$ matrices where each one of the
$\lambda$ columns is an element of $A^n$.
Now, for each $n$-ary operation $t\in\clo{C}$ define
\begin{equation*}
  E_t :=\{[a_i]_{i<\lambda}\in (A^n)^\lambda:
  f(a_i)=t(a_i)\ \text{for all $i<\lambda$}\}
\end{equation*}
to be the set of all $n\times\lambda$ matrices 
where $f$ is equal to $t$.
Hence,
\begin{equation*}
N_t :=(A^n)^\lambda\setminus E_t
\end{equation*}
is the set of all $n\times\lambda$ matrices 
where $f$ is not equal to $t$.

Let $\mathcal{F}_\lambda$
denote the collection of all subsets
of $(A^n)^\lambda$ of the form $N_t$ ($t\in\clo{C}$) defined above.
We will say that $\mathcal{F}_\lambda$  
has the \emph{finite intersection property} if the intersection of any
finite subfamily of $\mathcal{F}_\lambda$ is nonempty.

\begin{lm}
  \label{lm-fip}
  Let $\clo{C}$ be a clone on a set $A$, and let
  $f\colon A^n\to A$ be an $n$-ary operation on $A$ $(0<n<\omega)$.
  The following conditions are equivalent for every
  nonzero $\lambda<\omega$.
  \begin{enumerate}
  \item[{\rm(i)}]
      The condition below from Theorem~\ref{thm-char-lambda}:
      \begin{enumerate}
      \item[$(\ddagger)_\lambda$]
    $A^n\,\bigl(=\dom(f)\bigr)$ has a finite cover
    $\CCf_\lambda\,\bigl(\subseteq \PP(A^n)\bigr)$ such that
    whenever $\mathcal{B}\subseteq\CCf_\lambda$
    satisfies $|\mathcal{B}|\le\lambda$, there exists an $n$-ary
    $t^{[\mathcal{B}]}\in\clo{C}$ 
    such that $f|_{\bigcup{\mathcal{B}}}=t^{[\mathcal{B}]}|_{\bigcup{\mathcal{B}}}$.
      \end{enumerate}  
    \item[{\rm(ii)}]
     $(A^n)^\lambda$ has a finite cover
     $\mathcal{D}_\lambda\,\bigl(\subseteq\PP((A^n)^\lambda)\bigr)$
      such that for every $D\in\mathcal{D}_\lambda$
      there exists $s^{[D]}\in\clo{C}$ such that we have
      \begin{equation}
        \label{eq-ii}
      f(a_i)=s^{[D]}(a_i)
      \text{ for all }
      i<\lambda
      \text{ whenever }
      [a_i]_{i<\lambda}\in D.
      \end{equation}
    \item[{\rm(iii)}]
      $\mathcal{F}_\lambda$ fails to have the finite intersection property.
  \end{enumerate}
\end{lm}  

\begin{proof}
  Suppose (i) holds. Since $\CCf_\lambda$ is finite, so is
  \[
  \mathcal{D}_\lambda:=\Bigl\{\Bigl(\bigcup \mathcal{B}\Bigr)^\lambda:
  \mathcal{B}\subseteq\CCf_\lambda,\,
  |\mathcal{B}|\le\lambda\Bigr\}.
  \]
  Since $\CCf_\lambda$ covers $A^n$, it follows that 
  $\mathcal{D}_\lambda$ covers $(A^n)^\lambda$.
  Moreover, our assumption $(\ddagger)_\lambda$ yields that for every
  member $D=(\bigcup \mathcal{B})^\lambda$ of $\mathcal{D}_\lambda$ the operation
  $s^{[D]}:=t^{[\mathcal{B}]}\in\clo{C}$ satisfies the requirement in (ii).
  This finishes the proof of (i) $\Rightarrow$ (ii).

  Conversely, assume (ii), and for each $D\in\mathcal{D}_\lambda$
  and each $j<\lambda$ define $D^{(j)}$ to be the projection of $D$ onto
  its $j$-th coordinate; that is, $D^{(j)}:=\{a_j:[a_i]_{i<\lambda}\in D\}$.
  Furthermore, let $\tilde{D}:=\bigcup\{D^{(j)}:j<\lambda\}$.
  Condition~\eqref{eq-ii} from assumption (ii) implies that for each
  $\tilde{D}$ with $D\in\mathcal{D}_\lambda$,
    \begin{equation}
      \label{eq-tildeD}
    \text{$s^{[D]}\in\clo{C}$\ \ satisfies }  
    f(a)=s^{[D]}(a)\ \text{for all}\  a\in \tilde{D}.
    \end{equation}

  Since $D\subseteq\tilde{D}^\lambda$ for every $D\in\mathcal{D}_\lambda$
  and $\mathcal{D}_\lambda$
  is a finite cover of $(A^n)^\lambda$,
  it follows that
  the set $\mathcal{E}:=\{D^{(j)}:D\in\mathcal{D}_\lambda,\,j<\lambda\}$
  is 
  a finite cover of $A^n$.
  Let $\mathcal{A}$ denote the Boolean algebra of sets
  generated by $\mathcal{E}$. Clearly, $\mathcal{A}$ is finite, and
  the set $\CCf_\lambda$ of all atoms of $\mathcal{A}$ is a finite cover
  of $A^n$ which partitions $A^n$ into nonempty subsets.
  Our goal is to show that $\CCf_\lambda$ satisfies the requirements in
  condition $(\ddagger)_\lambda$.

  Let $\mathcal{B}=\{C_0,\dots,C_{\lambda-1}\}$ be any subset of
  $\CCf_\lambda$ of size $\le\lambda$. Choose $a_i\in C_i$
  for each $i<\lambda$. Since $[a_i]_{i<\lambda}\in (A^n)^\lambda$, there must
  exist $D\in\mathcal{D}_\lambda$ with $[a_i]_{i<\lambda}\in D$.
  Hence, $[a_i]_{i<\lambda}\in\prod_{i<\lambda}(C_i\cap D^{(i)})$, showing that
  each $C_i\cap D^{(i)}$ is a nonempty member of $\mathcal{A}$ contained in
  an atom $C_i$. This forces $C_i\subseteq D^{(i)}$ for all $i<\lambda$.
  Hence,
  \[
  C_0\times\dots\times C_{\lambda-1}\subseteq
  D^{(0)}\times\dots\times D^{(\lambda-1)}\subseteq\tilde{D}^\lambda,
  \]
  which implies that
  $\bigcup\mathcal{B}\subseteq\bigcup\{D^{(i)}:i<\lambda\}=\tilde{D}$.
  Thus, \eqref{eq-tildeD} implies that $f(a)=s^{[D]}(a)$ holds
  for all $a\in\bigcup\mathcal{B}$.
  This completes the proof of (ii) $\Rightarrow$ (i).

  It remains to prove that (ii) $\Leftrightarrow$ (iii).
  Condition (iii) holds, i.e.,
  $\mathcal{F}_\lambda$ fails to have the finite intersection property,
  if and only if $\clo{C}$ contains finitely many
  $n$-ary operations $t_1,\dots,t_r$
  such that $N_{t_1}\cap\dots\cap N_{t_r}=\emptyset$,
  or equivalently, $E_{t_1}\cup\dots\cup E_{t_r}=(A^n)^\lambda$.
  Thus, if (iii) holds, then (ii) also holds with the choice
  $\mathcal{D}_\lambda=\{E_{t_j}:j=1,\dots,r\}$.
  Conversely, if (ii) holds, then we have
  $D\subseteq E_{s^{[D]}}$ for every $D\in\mathcal{D}_\lambda$.
  Hence we have finitely many operations $s^{[D]}$ ($D\in\mathcal{D}_\lambda$)
  in $\clo{C}$ such that
  $\bigcup\{E_{s^{[D]}}:D\in\mathcal{D}_\lambda\}=(A^n)^\lambda$. 
  As we explained at the beginning of this paragraph, this proves (iii).
\end{proof}

\begin{proof}[Proof of Theorem~\ref{thm-char-lambda}]
Let $\clo{C}$ be a clone on a set $A$, and let
$f\colon A^n\to A$ be an $n$-ary operation on $A$ $(0<n<\omega)$.
Theorem~\ref{thm-char-lambda} states for every $\lambda<\omega$,
the property that
  \begin{enumerate}
  \item[$(\dagger)_\lambda$]
    $f$ is $\lambda$-ultrainterpolable by $\clo{C}$
  \end{enumerate}
  is characterized by the condition
  \begin{enumerate}
  \item[$(\ddagger)_\lambda$]
    $A^n\,\bigl(=\dom(f)\bigr)$ has a finite cover
    $\CCf_\lambda\,\bigl(\subseteq \PP(A^n)\bigr)$ such that
    whenever $\mathcal{B}\subseteq\CCf_\lambda$
    satisfies $|\mathcal{B}|\le\lambda$, there exists an $n$-ary
    $t^{[\mathcal{B}]}\in\clo{C}$ 
    such that $f|_{\bigcup{\mathcal{B}}}=t^{[\mathcal{B}]}|_{\bigcup{\mathcal{B}}}$.
  \end{enumerate}
  
The statement of the theorem is vacuously true for $\lambda=0$,
because
both conditions $(\dagger)_0$ and $(\ddagger)_0$ hold
for $f$ for the following reason. For $(\dagger)_0$, notice that
$f$ is $0$-ultrainterpolable by any $n$-ary projection in $\clo{C}$,
since any two $n$-ary operations (on any set) agree on $\emptyset$.
For $(\ddagger)_\lambda$, if we choose $\CCf_0:=\{A^n\}$, the same
reasoning yields the required equality for $\mathcal{B}=\emptyset$ and
any $n$-ary projection $t^{[\mathcal{B}]}$.

Therefore, we will assume from now on that $\lambda>0$.
First, we will prove $(\dagger)_\lambda\Rightarrow(\ddagger)_\lambda$.
To obtain a contradiction, assume  
that $f$ is $\lambda$-ultrainterpolable by $\clo{C}$, but
$(\ddagger)_\lambda$ fails.
By Lemma~\ref{lm-fip} the latter assumption means that the family
$\mathcal{F}_\lambda$ of subsets of $I:=(A^n)^{\lambda}$ 
has the finite intersection property.
It follows that there exists an
ultrafilter $\mathcal{U}$ on $I$ such that
$\mathcal{F}_\lambda\subseteq\mathcal{U}$.
Each member $\alpha\in I=(A^n)^\lambda $ is an $n\times \lambda$ matrix
$\alpha=[\alpha_j^{(\ell)}]_{j<n}^{\ell<\lambda}$.
For each $j<n$ and $\ell<\lambda$ define an element $\bar{a}_j^{(\ell)}$
of $A^I$ as follows:
$\bar{a}_j^{(\ell)}:=(\alpha_j^{(\ell)})_{\alpha\in I}$.
This yields a subset
\begin{equation}
  \label{eq-setS}
  S:=\{(\bar{a}_0^{(\ell)}/\mathcal{U},\dots,
  \bar{a}_{n-1}^{(\ell)}/\mathcal{U}):\ell<\lambda\}
\end{equation}
of $(A^I/\mathcal{U})^n$ with
$|S|\le\lambda$.

Our assumption is that $f$ is $\lambda$-ultrainterpolable by $\clo{C}$.
Hence, for the ultrafilter $\mathcal{U}$ and set
$S\subseteq(A^I/\mathcal{U})^n$ of size $\le\lambda$ just constructed,
$f_{\mathcal{U}}$ is interpolated on $S$ by $t_{\mathcal{U}}$ for some $n$-ary
operation $t\in\clo{C}$;
that is, $f_{\mathcal{U}}$ and $t_{\mathcal{U}}$ satisfy
\begin{equation}
  \label{eq-fU-tU}
f_{\mathcal{U}}(\bar{a}_0^{(\ell)}/\mathcal{U},\dots,\bar{a}_{n-1}^{(\ell)}/\mathcal{U})
=t_{\mathcal{U}}(\bar{a}_0^{(\ell)}/\mathcal{U},\dots,\bar{a}_{n-1}^{(\ell)}/\mathcal{U})
\quad\text{for all $\ell<\lambda$}.
\end{equation}
Thus, the set
\[
E:=\{\alpha\in I=(A^n)^\lambda :
f(\alpha_0^{(\ell)},\dots,\alpha_{n-1}^{(\ell)})
=t(\alpha_0^{(\ell)},\dots,\alpha_{n-1}^{(\ell)})
\text{ for all $\ell<\lambda$}\}
\]
is a member of $\mathcal{U}$.
Clearly, $E\subseteq E_t$, so $E_t\in\mathcal{U}$. 
However, by the construction of $\mathcal{U}$ we have that
$N_t=I\setminus E_t\in \mathcal{F}_\lambda \subseteq\mathcal{U}$, so
$E_t\notin\mathcal{U}$. This contradiction
finishes the proof of
$(\dagger)_\lambda\Rightarrow(\ddagger)_\lambda$.

To prove the implication $(\ddagger)_\lambda\Rightarrow(\dagger)_\lambda$,
assume that $(\ddagger)_\lambda$ holds,
let $A^I/\mathcal{U}$ be an arbitrary ultrapower of $A$, and
consider a subset $S$ of $(A^I/\mathcal{U})^n$ of size $\le\lambda$.
Although the set $I$ is now different from the set $I$ in the preceding
paragraphs, we may write $S$ in
the form \eqref{eq-setS}
where $\bar{a}_j^{(\ell)}=(a_{ji}^{(\ell)})_{i\in I}\in A^I$
for all $j<n$ and $\ell<\lambda$.
We have to show that there exists an $n$-ary operation $t\in\clo{C}$
such that $t_{\mathcal{U}}$ interpolates $f_{\mathcal{U}}$ on $S$, i.e., 
such that \eqref{eq-fU-tU} holds.

Let $\CCf_\lambda=\{C_0,\dots, C_{r-1}\}$ be a finite cover of $A^n$
(of size $r$)
provided by our assumption $(\ddagger)_\lambda$; i.e., $\CCf_\lambda$
has the property that whenever $\mathcal{B}\subseteq\CCf_\lambda$
satisfies $|\mathcal{B}| \le \lambda$, there exists an $n$-ary
$t^{[\mathcal{B}]} \in\clo{C}$ such that
$f|_{\bigcup\mathcal{B}} = t^{[\mathcal{B}]}|_{\bigcup\mathcal{B}}$.
For each $\lambda$-tuple 
$\epsilon=(\epsilon_0,\dots,\epsilon_{\lambda-1})\in r^\lambda$
of subscripts of members of $\CCf_\lambda$ define
\[
I_\epsilon:=\{i\in I: (a_{0i}^{(\ell)},\dots,a_{n-1,i}^{(\ell)})\in C_{\epsilon_\ell}
\text{ for all $\ell<\lambda$}\}.
\]
These sets form a finite cover
$\mathcal{I}:=\{I_\epsilon:\epsilon\in r^\lambda\}$
of $I$ (with possibly some of the sets $I_\epsilon$ empty).
Since
$\mathcal{U}$ is an ultrafilter on $I$, there exists $\epsilon\in r^\lambda $
such that $I_\epsilon\in\mathcal{U}$.
Now let $\mathcal{B}:=\{C_{\epsilon_0},\dots,C_{\epsilon_{\lambda-1}}\}$.
We have $\mathcal{B}\subseteq\CCf_\lambda$ and $|\mathcal{B}|\le\lambda$,
therefore there is a corresponding $n$-ary operation
$t^{[\mathcal{B}]}\in\clo{C}$ satisfying
$f|_{\bigcup\mathcal{B}} = t^{[\mathcal{B}]}|_{\bigcup\mathcal{B}}$.
Since
$C_{\epsilon_0}\times\dots\times C_{\epsilon_{\lambda-1}}
\subseteq(\bigcup\mathcal{B})^\lambda$,
it follows that the set
\[
\{i\in I:
f(a_{0i}^{(\ell)},\dots,a_{n-1,i}^{(\ell)})
=t^{[\mathcal{B}]}(a_{0i}^{(\ell)},\dots,a_{n-1,i}^{(\ell)})
\text{ for all $\ell<\lambda$}\}
\]
contains $I_\epsilon$, and hence belongs to $\mathcal{U}$.
This establishes \eqref{eq-fU-tU} for $t:=t^{[\mathcal{B}]}$,
and hence completes the proof of
$(\ddagger)_\lambda\Rightarrow(\dagger)_\lambda$.
\end{proof}

\section{Clones containing near unanimity operations}

Recall that for any integer $d\ge3$,
a $d$-ary operation $h$ on a set $A$ is called a
$d$-ary \emph{near unanimity operation} if it satisfies
\begin{equation}
  \label{eq-nu}
h(a,\dots,a,\overbrace{b}^{\text{$i$-th}},a,\dots,a)=a
\qquad\text{for all $a,b\in A$ and $1\le i\le d$,}
\end{equation}
where the sole occurrence of the letter $b$ is in the $i$-th position.

In \cite{vaggione}, Vaggione proved the following infinitary version of the
Baker--Pixley Theorem: Let $\clo{C}$ be the clone of term operations of
an algebra $\al{A}$ with universe $A$, and assume that
$\clo{C}$ contains a $d$-ary
near unanimity operation. If $f$ is an operation on $A$ such that
for every ultrafilter $\mathcal{U}$ on any set $I$,
\begin{enumerate}
\item[$(\diamond)$]
  the extension
  $f_{\mathcal{U}}$ of $f$ to $A^I/\mathcal{U}$ preserves all subalgebras
  of $(\al{A}^I/\mathcal{U})^{d-1}$,
\end{enumerate}
then $f\in\clo{C}$.

Since the clone of term operations of $\al{A}$ is $\clo{C}$, 
the clone of term operations of the ultrapower
$\al{A}^I/\mathcal{U}$ is $\clo{C}_{\mathcal{U}}$.
Therefore, by Lemma~\ref{lm-invrel}(3), $(\diamond)$
is equivalent to the condition that
$f_{\mathcal{U}}$ is $(d-1)$-interpolable by operations in $\clo{C}_{\mathcal{U}}$.
Since $(\diamond)$ is required to hold for every ultrafilter $\mathcal{U}$,
the assumption on $f$ in Vaggione's result is equivalent to
saying that $f$ is $(d-1)$-ultrainterpolable by $\clo{C}$.
Hence, Vaggione's main result in \cite{vaggione}
states, in our terminology, that every clone that contains a $d$-ary near
unanimity operation is $d$-ultraclosed.
We now derive this result from Corollary~\ref{cor-char-kappa}.

\begin{thm}[\cite{vaggione}]
  \label{thm-nu}
  Every clone that contains a $d$-ary near unanimity operation $(d\ge3)$ is
  $d$-ultraclosed, and hence is also ultralocally closed.
\end{thm}  

\begin{proof}
Let $\clo{C}$ be a clone on a set $A$ such that $\clo{C}$ contains a $d$-ary
near unanimity operation $h$ ($d\ge3$). Our goal is to show that
$\clo{C}=\ULC_d(\clo{C})$. By Lemma~\ref{lm-incl}(2), this will also imply
that $\clo{C}=\ULC_\omega(\clo{C})$.
By Corollary~\ref{cor-char-kappa}, to establish $\clo{C}=\ULC_d(\clo{C})$,
it suffices to prove that every operation $f\colon A^n\to A$ 
$(0<n<\omega)$ which satisfies
condition $(\ddagger)_{d-1}$ from
Theorem~\ref{thm-char-lambda} is a member of $\clo{C}$.
So, assume that condition $(\ddagger)_{d-1}$ holds for $f$.
Thus, there is a finite cover
$\CCf_{d-1}$ of $A^n$ and
there exist $n$-ary operations
$t^{[\mathcal{B}]}\in\clo{C}$ for every set $\mathcal{B}\subseteq\CCf_{d-1}$
with $|\mathcal{B}|\le d-1$ such that
$f|_{\bigcup\mathcal{B}}=t^{[\mathcal{B}]}|_{\bigcup\mathcal{B}}$.
If $|\CCf_{d-1}|\le d-1$, the last equality holds for
$\bigcup{\CCf_{d-1}}=A^n$, so
$f=t^{[\CCf_{d-1}]}\in\clo{C}$.

Assume from now on that $|\CCf_{d-1}|\ge d$.
We will apply the usual Baker--Pixley argument to the regions in
$\CCf_{d-1}$ to show, by induction on $m$, that for every $m\ge d-1$,
\begin{enumerate}
\item[$(*)_m$]
  $f|_{\bigcup\mathcal{B}}=t^{[\mathcal{B}]}|_{\bigcup\mathcal{B}}$ for some $n$-ary
  operation $t^{[\mathcal{B}]}\in\clo{C}$,
  whenever $\mathcal{B}\subseteq\CCf_{d-1}$ with $|\mathcal{B}|\le m$.
\end{enumerate}
This will complete the proof, because then by choosing $m=|\CCf_{d-1}|$
and $\mathcal{B}=\CCf_{d-1}$, we will have $\bigcup\CCf_{d-1}=A^n$ and
hence $f=t^{[\CCf_{d-1}]}\in\clo{C}$. 

To prove $(*)_m$ for $m\ge d-1$, notice first that $(*)_{d-1}$ is exactly
the condition that is forced by $(\ddagger)_{d-1}$. Assume therefore that
$m\ge d$ and $(*)_{m-1}$ holds. Let $\mathcal{B}=\{C_0,\dots,C_{m-1}\}$
be a subset of $\CCf_{d-1}$ of cardinality $\le m$. 
For each $i<m$, let $\mathcal{B}_i:=\mathcal{B}\setminus\{C_i\}$.
By the induction hypothesis $(*)_{m-1}$,
there exist $n$-ary operations
$t^{[\mathcal{B}_i]}\in\clo{C}$ such that
\begin{equation}
  \label{eq-induc-hyp}
  f|_{\bigcup\mathcal{B}_i}=t^{[\mathcal{B}_i]}|_{\bigcup\mathcal{B}_i}
  \qquad\text{for every $i<m$.}
\end{equation}
We claim that the operation
\begin{equation}
  \label{eq-new-op}
  t^{[\mathcal{B}]}:=h(t^{[\mathcal{B}_1]},\dots,t^{[\mathcal{B}_d]})\in\clo{C}
\end{equation}
satisfies the equality
\begin{equation}
  \label{eq-to-prove}
  f|_{\bigcup\mathcal{B}}=t^{[\mathcal{B}]}|_{\bigcup\mathcal{B}}
\end{equation}  
required by $(*)_m$.
Indeed, if $a\in\bigcup\mathcal{B}$, then $a\in C_j$ for some $j<m$,
so $a\in\bigcup\mathcal{B}_i$ for all
$i<m$ with $i\not=j$.
Thus, by \eqref{eq-induc-hyp}, $t^{[\mathcal{B}_i]}(a)=f(a)$ for
all $i<m$, $i\not=j$. Hence,
when evaluating the operation on the right hand side of
\eqref{eq-new-op} at $a$, all but possibly one of the arguments of
$h$ are equal to $f(a)$, therefore
the near unanimity identities in \eqref{eq-nu}
force $t^{[\mathcal{B}]}(a)=f(a)$. This proves \eqref{eq-to-prove}, and finishes
the proof of the theorem.
\end{proof}  

\section{Simple Modules}\label{Section5}

Our goal in this section is to prove that the clone of any simple module
is ultralocally closed. We do not know whether simplicity is a necessary
hypothesis for this result. 

\begin{thm}
  \label{thm-simple-module}
  The clone of any simple module is $4$-ultraclosed, and hence is
  also ultralocally closed.
\end{thm}

\begin{proof}
Let ${}_RA$ be a simple $R$-module, and let $\clo{C}$ denote its clone.
It follows from Corollary~\ref{cor-specprop2}  
that for all $\kappa\ge4$,
the $\kappa$-closure $\LC_\kappa(\clo{C})$ as well as 
the $\kappa$-ultraclosure $\ULC_\kappa(\clo{C})$ of $\clo{C}$
are clones of 
modules on the set $A$ 
which share
the underlying abelian group $\hat{A}$ of ${}_RA$.
Therefore, to determine
  these clones it suffices to determine the 
  rings of scalars of the corresponding modules.
  Let $\bar{R}$ and $S$ denote the scalar rings of the modules
  with clones $\ULC_4(\clo{C})$ and $\LC_\omega(\clo{C})$, respectively.
  We may assume without loss of generality that the actions of the rings
  $R$, $\bar{R}$, and $S$ are faithful, and identify each scalar in $R$, 
  $\bar{R}$, or $S$ with its action as an endomorphism of the underlying
  abelian group $\hat{A}$. Upon this identification
  $R$, $\bar{R}$, and $S$ become the set of all unary operations
  in $\clo{C}$, $\ULC_4(\clo{C})$, and $\LC_\omega(\clo{C})$, respectively.
  Hence $R\subseteq \bar{R}$, $R\subseteq S$, and showing that
  $C$ is $4$-ultraclosed amounts to showing that $R=\bar{R}$. 
  
  It follows from Jacobson's Density Theorem that the scalar ring $S$
  of the local closure $\LC_\omega(\clo{C})$
  of $\clo{C}$ 
  is the double centralizer ring of $R$. As a reminder,
  if ${}_{R}A$ is a simple left $R$-module
  and $D=\End({}_RA)$ is the (single) centralizer ring,
  then by Schur's Lemma, $D$ is a division ring. We let $D$ act on $A$
  on the right,
making $A_D$ a right $D$-vector space. The double centralizer
ring is the ring $\End(A_D)$ of $D$-linear
maps, which will act on the left.
It is clear that $R\subseteq\End(A_D)$.
The Density Theorem
asserts that the ring $R$ of $D$-linear maps is dense in
the ring $\End(A_D)$ of all $D$-linear maps
in the sense that every map $f\in\End(A_D)$ can be interpolated on each
finite subset of $A_D$ by a map in $R$.
In our language this asserts
that the local closure $\LC_\omega(\clo{C})$ of the clone
$\clo{C}$ of ${}_R A$ is the clone of the module
${}_{\End(A_D)} A$. Thus, $S=\End(A_D)$.

Next we want to show that $\Lambda_\omega(\clo{C})=\Lambda_4(\clo{C})$.
Let $\mathcal{R}$ be the set consisting of the following relations on $A$:
the graph $\gamma(+)$ of the binary operation $+$
(addition of the module ${}_R A$), and the graphs $\gamma(d)$
of all unary operations $d\in D$ (endomorphisms of the module ${}_S A$).
All relations in $\mathcal{R}$ have arity $\le3$, therefore
the clone $\clo{Pol}(\mathcal{R})$ is $4$-closed by Lemma~\ref{lm-invrel}(1).
Using the fact (see the proof of Corollary~\ref{cor-specprop2})
that $\clo{Pol}\bigl(\gamma(+)\bigr)$ is the clone of the module
${}_{\End(\hat{A})} A$ one can easily check that
$\clo{Pol}(\mathcal{R})$ coincides
with the clone of the module ${}_{\End(A_D)} A$.
Since the clone of ${}_{\End(A_D)} A$ is $\Lambda_\omega(\clo{C})$,
we get that $\Lambda_\omega(\clo{C})$ is $4$-closed.
This implies that $\Lambda_\omega(\clo{C})=\Lambda_4(\clo{C})$,
as claimed.

It follows now from Lemma~\ref{lm-incl}(2) that
\[
\textrm{Clone}({}_RA)=
\clo{C}\;\subseteq\;
\ULC_4(\clo{C})\;\subseteq\; \LC_4(\clo{C})\;=\;
\LC_\omega(\clo{C})=\textrm{Clone}({}_SA)
\]
where the leftmost term is the clone of ${}_RA$ and
the rightmost term is the clone of ${}_SA$,
$S=\End(A_D)$.
Hence, for the unary components of these clones we have that
$R\subseteq\bar{R}\subseteq S$.
Consequently,
to establish that $\clo{C}$ is $4$-ultraclosed, i.e., $R=\bar{R}$,
it remains to show for every $D$-linear map $f\in S=\End(A_D)$ that
if $f$ is in the $4$-ultraclosure of $\clo{C}$, then $f\in R$.
There is nothing to prove if the set
$A$ is finite, because then $\clo{C}=\LC_\omega(\clo{C})$
(see Lemma~\ref{lm-incl}(3)), and hence by the last displayed line
$\clo{C}=\ULC_4(\clo{C})$.

Assume from now on that $A$ is infinite, 
let $f\in S=\End(A_D)$, and suppose $f$ is in the $4$-ultraclosure of
$\clo{C}$. Our goal is to prove that $f\in R$.
We will apply to $f$
the criterion of Corollary~\ref{cor-char-kappa} for $\kappa=4$
in the case $n=\lambda=1$. By condition $(\ddagger)_1$, for $n=1$,
the set $A$ has a finite cover ${\mathcal C}_1=\{B_0,\ldots,B_{m-1}\}$
such that whenever $B_i\in {\mathcal C}_1$ ($i<m$), there is
an element $r_i\in R$ that interpolates $f$
on $B_i$ (that is, $f|_{B_i}=r_i|_{B_i}$).
Since $f$ and $r_i$ are both $D$-linear mappings,
the kernel of $f-r_i$ is a $D$-subspace of $A_D$
containing $B_i$. Hence, we may enlarge
each set $B_i$ to $B_i'=\ker(f-r_i)$ and still have a
finite cover
$\{B_0',\ldots,B_{m-1}'\}$
of $A$ such that
$f|_{B_i}=r_i|_{B_i}$ for each $i<m$, but now we have
that the sets $B_i'$ ($i<m$) are $D$-subspaces of $A_D$.
Replacing each $B_i$ with $B_i'$ and dropping the primes,
we now assume that our original
set ${\mathcal C}_1$ consisted of $D$-subspaces of $A_D$.

We may, in fact, assume more. Recall that our goal is to
prove that the $D$-linear map
$f$ is in $R$.
But the $D$-linear map
$f$ is in $R$
iff the $D$-linear map $f-r_0$ is in $R$.
Therefore, we may replace
each of $f, r_0, r_1, \ldots, r_{m-1}$ with
$f':=f-r_0, r_0':=r_0-r_0, r_1':=r_1-r_0, \ldots, r_{m-1}':=r_{m-1}-r_0$
and prove the desired statement in the setting where
the first scalar $r_0'=r_0-r_0$ is zero.
Dropping the primes we henceforth assume
that $f|_{B_i}=r_i|_{B_i}$
for all $i<m$,
and that the first ring element
on the list, $r_0$, equals $0$.

If $D$ is infinite, then there is nothing more to do. It is known
that a vector space $A_D$ over an infinite division ring $D$
cannot be expressed as a finite union of proper subspaces,
so $A=B_j$ must hold for some $j<m$. In this case,
$f=f|_A=r_j|_A=r_j$, so $f\in R$,
as desired.

Henceforth we assume that $D$ is a finite field.
Since the vector space $A_D$ is infinite, $A_D$ must be infinite
dimensional.
In this situation we use Neumann's Lemma \cite{neumann1, neumann2},
which asserts
that if a group $G$ is expressible as a finite, irredundant
union of cosets of subgroups, $G=\bigcup_{i<n} g_iH_i$,
then the index $\bigl[G: \bigcap_{i<n} H_i\bigr]$ is finite.
Here we take $G=A$ and $g_iH_i=B_i$ to obtain 
(after discarding some of the $B_i$'s, if the
cover ${\mathcal C}_1$ is redundant) that
the intersection $I:=\bigcap B_i$ is a $D$-subspace of $A$ that has
finite group-theoretic index in $A$.
Since $f|_I=r_0|_I=\dots=r_{m-1}|_I$ and $r_0=0$, 
we derive that
each of the $D$-linear maps
$f, r_0,\ldots, r_{m-1}$ contains $I$ in its kernel.
Since $I$ has finite group-theoretic index in $A$,
the images of the maps $f, r_0, \ldots, r_{m-1}$
are all finite. In particular, the $D$-subspaces
$r_0A, \ldots, r_{m-1}A$ are finite subspaces
of the infinite dimensional $D$-space $A_D$.

Choose $m$ independent subspaces of $A$, $V_0, \ldots, V_{m-1}$,
for which there exist
$D$-linear isomorphisms
$\sigma_i\colon r_iA\to V_i$ ($i<m$).
This is possible
since each $r_iA$ is a finite dimensional
subspace of the infinite dimensional space $A_D$.
By the facts that $R$ is dense in $S=\End(A_D)$
and that each $r_iA$ ($i<m$)
is finite dimensional,
there exist $s_i\in R$ such that 
$s_i|_{r_iA}=\sigma_i$
for all $i<m$.
Consider the ring element $t = s_0r_0+\cdots+s_{m-1}r_{m-1}$
in $R$.

\begin{clm}
  \label{clm-kernel}
  The $D$-linear 
  map $t$
  has kernel contained in $\ker(f)$.
\end{clm}

\begin{proof}[Proof of Claim~\ref{clm-kernel}.]
Choose a vector $v\in A$
and assume that $0=tv = \sum_{i<m} s_ir_iv$.
Since the $s_i$'s have independent ranges, it follows that
$s_ir_iv = 0$ for all $i<m$. But since $s_i$ is an isomorphism
defined on the range of $r_i$, we even get that $r_iv=0$
for all $i<m$. This implies that $v\in \bigcap_{i<m} \ker(r_i)$.
Now, since $v\in A = \bigcup_{i<m} B_i$, there is some $i<m$
such that $v\in B_i$, and for this $i$ we have
$f(v)=r_iv = 0$. Hence $v\in\ker(f)$.
\qedsymbdiamond
\end{proof}

At this point we know that $t$ and $f$ are $D$-linear endomorphisms
of the space $A_D$, and that $\ker(t)\subseteq \ker(f)$.
It follows from the First Isomorphism Theorem
of linear algebra that there is a $D$-linear
map $u$ such that $ut=f$. Since the image of $t$,
$tA \subseteq \sum_{i<m} V_i$, is finite dimensional,
the Density Theorem allows us to interpolate $u$ on $tA$
by an element $u'\in R$. In fact, since $u'$ is itself
$D$-linear, there is no harm in assuming that $u=u'$, so that
$u\in R$. With this choice $f=ut\in R$.

To summarize, we argued that if an operation
$f\colon A\to A$ belongs to the unary component
of the $4$-ultraclosure of 
the clone of ${}_RA$, then in fact $f$
equals an operation in the unary component
of the clone of ${}_RA$. This establishes that the clone
of ${}_RA$ is $4$-ultraclosed. By Lemma~\ref{lm-incl}(2) it follows also
that the clone of ${}_R A$ is ultralocally closed.
\end{proof}

\section{$\LC_\omega$ versus $\ULC_\omega$} \label{Section6}

In this final section we discuss some similarities and dissimilarities
between local closure and ultralocal closure.
Since both $\LC_\omega$ and $\ULC_\omega$
equal the identity operator on
the lattice of clones on a finite set,
we will assume throughout that the base set $A$ is infinite.

It is known (see, e.g., \cite{rosenberg}, \cite[p.~367]{goldstern-pinsker}, or
Subsection~6.C below) that there are
$2^{2^\nu}$ clones on an infinite set $A$ of cardinality $\nu$.
Among these, only $2^\nu$ are locally closed
(see, e.g.~\cite[p.~396]{goldstern-pinsker}), which shows that the range
of the closure operator $\LC_\omega$ on the lattice of clones
on $A$ is small.
One of our goals in this section is to prove the theorem below,
which shows that, in contrast to
$\LC_\omega$, the range of the closure operator $\ULC_\omega$ on
the lattice of clones on $A$
is large.

\begin{thm}
  \label{thm-lots}
If $A$ is an infinite set of cardinality $\nu$, then the lattice of clones on
$A$ contains
\begin{enumerate}
\item[{\rm(1)}]
  an interval $[\clo{C}_1,\clo{D}_1]$ of size $2^{2^\nu}$ 
  such that every clone in the interval is ultralocally closed, and
\item[{\rm(2)}]
  an interval $[\clo{C}_2,\clo{D}_2]$ of size $2^{2^\nu}$ 
  such that no clone in the interval is ultralocally closed.
\end{enumerate}
In fact, the interval $[\clo{C}_1,\clo{D}_1]$ can be chosen so that
$\clo{D}_1=\clo{O}_A$, $\clo{C}_1$ is generated by a single operation, and
the interval contains $2^{2^\nu}$ clones that are maximal in
$\clo{D}_1=\clo{O}_A$.
The interval $[\clo{C}_2,\clo{D}_2]$ can be chosen so that it is
isomorphic to the lattice of all clones on $A$, hence it also contains
$2^{2^\nu}$ clones that are maximal in $\clo{D}_2$.
\end{thm}

Another well-known fact
(noted, e.g., in \cite[p.~395]{goldstern-pinsker})
is that
if $A$ is an
infinite set, then the
lattice of all locally closed clones on $A$
is not algebraic. Equivalently, the
closure operator
\[
\LC_\omega\langle-\rangle\colon
\PP(\clo{O}_A)\to\PP(\clo{O}_A),\ \
F\mapsto\LC_\omega(\langle F\rangle)
\]
on $\clo{O}_A$, which assigns to each set of operations the least
locally closed clone containing it, is not an algebraic closure operator.
Here we say
that a closure operator on a set $S$ is \emph{algebraic} if
for any set $X\subseteq S$, $X$ is closed if and only if $X$
is the set-theoretic union of
the closures of its finite subsets.

Analogously, given an infinite cardinal $\kappa$, we will say that 
a closure operator 
$\bar{\phantom{b}}\colon \PP(S)\to\PP(S)$, $X\mapsto\bar{X}$
on $S$ is \emph{$\kappa$-algebraic} if
for any set $X\subseteq S$,
\begin{equation}
  \label{eq-def-kappa-alg}
X=\bar{X}\ \ \Leftrightarrow\ \
X=\bigcup\{\bar{Y}:Y\subseteq X,\,|Y|<\kappa\}.
\end{equation}
So, a closure operator $\bar{\phantom{b}}$ on $S$
is $\kappa$-algebraic
if for any set $X\subseteq S$, $X$ is closed if and only if $X$
is the union of the closures of its subsets of size less than $\kappa$.
In this terminology `algebraic'
is the same as `$\omega$-algebraic'.

\begin{thm}
  \label{thm-omega1-alg}
For arbitrary infinite set $A$, the closure operator
\[
\ULC_\omega\langle-\rangle\colon
\PP(\clo{O}_A)\to\PP(\clo{O}_A),\ \
F\mapsto\ULC_\omega(\langle F\rangle),
\]  
which assigns to each set of operations on $A$ the least ultralocally closed
clone containing it,
\begin{enumerate}
\item[{\rm(1)}]
  is not algebraic, but
\item[{\rm(2)}]
  it is $\omega_1$-algebraic.
\end{enumerate}
Thus, a clone $\clo{C}$ on $A$ is ultralocally closed if and only if
$\clo{C}$ contains the ultralocal closure of every countably generated
subclone of $\clo{C}$.
\end{thm}

Of course, for each set $A$, the
local closure operator $\LC_\omega\langle-\rangle$
on $\clo{O}_A$ is $\kappa$-algebraic for large enough $\kappa$,
say for $\kappa>2^{|A|}$, because every clone on $A$ has size $\le 2^{|A|}$.
But there is no fixed $\kappa$ for which the
local closure operator $\LC_\omega\langle-\rangle$ is $\kappa$-algebraic
for all infinite $A$, as the next theorem asserts.

\begin{thm}
  \label{thm-not-kappa-alg}
If $A$ is an infinite set of cardinality $\nu$, then
the closure operator $\LC_\omega\langle-\rangle$ is not $\kappa$-algebraic
for any infinite regular cardinal $\kappa\le\nu$.
\end{thm}

Our last theorem answers a question posed by the referee:
``{\it Is it obvious that the full clone on a countable
  set is (or is not)
  the ultralocal closure of a finite (countable) clone?}''
The theorem implies that the full clone on a countable set is not
the ultralocal closure of any set of functions
of cardinality less than $2^\omega$.

\begin{thm}
  \label{thm-referee}
  Let $A$ be an infinite set of cardinality $\nu$, and
  let $F\subseteq\clo{O}_A$.
  If $\ULC_\omega(\langle F\rangle)$ is an uncountable
  clone that contains a near unanimity operation, 
  then $|F|=|\ULC_\omega(\langle F\rangle)|$.
  In particular, if $\ULC_\omega(\langle F\rangle)=\clo{O}_A$, then
  $|F|=2^\nu$.
\end{thm}

Before proving these results in Subsection~6.D, 
we discuss some examples.

\bigskip

\noindent
6.A. {\bf Alternating groups and their clones. }
For an arbitrary set $A$ and 
for any permutation $\pi$ of $A$ the \emph{support} of $\pi$
is defined to be the set 
$\supp(\pi):=\{a\in A: \pi(a)\not=a\}$.
We will denote 
the group of all permutations of $A$ of finite support by
$\Sym_\omega(A)$. The \emph{alternating group on $A$} is the subgroup
$\Alt(A)$ of $\Sym_\omega(A)$ consisting of all
even permutations.
The essentially unary clones generated by the groups $\Alt(A)$
and $\Sym_\omega(A)$ will be denoted by
$\clo{Alt}(A)$ and $\clo{Sym}_\omega(A)$, respectively.
They are different when $|A|>1$, since $\clo{Sym}_\omega(A)$ will contain odd
permutations of $A$ and 
$\clo{Alt}(A)$ will not.

The next theorem describes an example of a clone that is not
ultralocally closed.
This example will also play role in the proof of
Theorem~\ref{thm-lots}(2).

\begin{thm}
  \label{thm-alt}
  If $A$ is an infinite set, then the clone $\clo{Alt}(A)$
  is not ultralocally closed.
  Its ultralocal closure is the clone $\clo{Sym}_\omega(A)$.
  In fact,
  \begin{enumerate}
  \item[{\rm(i)}]
    $\ULC_d(\clo{Alt}(A))=\clo{Sym}_\omega(A)$\ \
    for all\ \ $4\le d\le\omega$;
    while
  \item[{\rm(ii)}]
    $\LC_d(\clo{Alt}(A))$ is the essentially unary clone generated by
    the monoid of all injective unary operations $A\to A$,
    for all\ \ $4\le d\le\omega$.
  \end{enumerate}
\end{thm}  

\begin{proof}
  The first two sentences of the claim,
  which assert that $\clo{Sym}_\omega(A)$
  is the ultralocal closure of $\clo{Alt}(A)$,
  follow from (i) when $d=\omega$.
  To prove (i)--(ii), 
  fix $d$ such that $4\le d\le\omega$. 
  It follows from Corollary~\ref{cor-specprop1}(1)--(2) that
  both clones $\ULC_d(\clo{Alt}(A))$  and $\LC_d(\clo{Alt}(A))$  
  are essentially unary,
  and every unary operation $f\colon A\to A$ in them is injective.
  Thus, in both statements (i)--(ii), the clone equalities follow if we
  establish that the clones involved contain the same
  injective unary operations $A\to A$.
  
  Now, to finish the proof of (ii), it is enough to observe that
  every injective unary operation $A\to A$ is $k$-interpolable by permutations
  in $\Alt(A)$ for every $k<d$. 

  For the proof of (i) recall that our assumption $d\le\omega$ implies, by
  Lemma~\ref{lm-incl}(2), that
  $\ULC_d(\clo{Alt}(A))\supseteq\ULC_\omega(\clo{Alt}(A))$. 
  Hence, the equality in (i) will follow if we prove that
  for all injective unary operations $f\colon A\to A$,
  \begin{equation}\label{eq-alt}
    f\in\ULC_d(\clo{Alt}(A))
    \ \ \Rightarrow\ \ 
    f\in\Sym_\omega(A)
    \ \ \Rightarrow\ \ 
    f\in\ULC_\omega(\clo{Alt}(A)).
  \end{equation}
  
  To prove the first implication in \eqref{eq-alt} assume that 
  $f\in\ULC_d(\clo{Alt}(A))$ is injective.
  Applying Corollary~\ref{cor-char-kappa} with $\kappa=d$ and $\lambda=1$
  we see that
  $A$ has a finite cover $\CCf_1$ with the property that for each
  $C\in\CCf_1$ there exists $t^{[C]}\in\Alt(A)$ such that
  $f|_C=t^{[C]}|_C$. Since $\CCf_1$ is finite and each $t^{[C]}$ has finite
  support, it follows that $f$ moves at most finitely many elements of $A$.
  Therefore, the injectivity of $f$ implies that $f\in\Sym_\omega(A)$.
  
  For the second implication in \eqref{eq-alt}
  we want to show that $\Sym_\omega(A)$ is contained in the set
  of unary members of $\ULC_\omega(\clo{Alt}(A))$.
  Since $\ULC_\omega(\clo{Alt}(A))$
  is closed under composition, and since $\Sym_\omega(A)$ is generated
  under composition by all transpositions, it suffices to verify that
$\ULC_\omega(\clo{Alt}(A))$ contains every transposition.  
  To conclude that the transposition $f=(a\,b)$ ($a,b\in A$, $a\not=b$)
  belongs to 
$\ULC_\omega(\clo{Alt}(A))$ we need to show,
by Corollary~\ref{cor-char-kappa}, that condition
$(\ddagger)_k$ from Theorem~\ref{thm-char-lambda} holds for all
$k<\omega$. 
  There is nothing to prove for $k=0$, so assume that $k$ is a
  positive integer. Choose $\CCf_k$ to be any partition of $A$ into
  $k+1$ blocks $C_0,C_1,\dots,C_k$ such that $a,b\in C_0$ and every block
  $C_i$ ($i\le k$)
  has size $\ge2$. Clearly, such a partition exists, since $A$ is infinite.
  For every $\mathcal{B}\subseteq\CCf_k$ with $|\mathcal{B}|\le k$
  we have $(a\,b)|_{\bigcup\mathcal{B}}=\id|_{\bigcup\mathcal{B}}$ if
  $C_0\notin\mathcal{B}$, and
  $(a\,b)|_{\bigcup\mathcal{B}}=(a\,b)(c\,d)|_{\bigcup\mathcal{B}}$ if
  $C_0\in\mathcal{B}$ and $c,d$ are distinct elements of some
  $C_i\notin\mathcal{B}$.
  This proves that $(a\,b)\in\ULC_\omega(\clo{Alt}(A))$, as claimed.
\end{proof}  

For every finite subset $B$ of $A$ let $\Alt_B(A)$ denote the subgroup of
$\Alt(A)$ consisting of all permutations $\pi\in\Alt(A)$ with
$\supp(\pi)\subseteq B$.
Let $\clo{Alt}_B(A)$ denote the essentially unary clone generated by
the group $\Alt_B(A)$.

We will now show that these clones $\clo{Alt}_B(A)$, unlike $\clo{Alt}(A)$,
are ultralocally closed.
We will use these clones in the proof of Theorem~\ref{thm-omega1-alg}(1).

\begin{lm}
  \label{lm-altB}
  If $B$ is a finite subset of an infinite set $A$,
  then the clone $\clo{Alt}_B(A)$
  is locally closed, and hence is ultralocally closed; that is,
  \[
  \clo{Alt}_B(A)=\ULC_\omega(\clo{Alt}_B(A))=\LC_\omega(\clo{Alt}_B(A)).
  \]
\end{lm}  

\begin{proof}
  By Corollary~\ref{cor-specprop1}(1), all three clones here are essentially
  unary. Hence, by Lemma~\ref{lm-incl}(1), it suffices to show for every
  unary operation $f\in\LC_\omega(\clo{Alt}_B(A))$ that $f\in\clo{Alt}_B(A)$.
  So, let $f\colon A\to A$ be a unary operation in $\LC_\omega(\clo{Alt}_B(A))$.
  Then $f$ is interpolated by a permutation $\pi_C\in\Alt_B(A)$ for any
  finite set $C=B\cup\{a\}$ where $a\in A\setminus B$.
  Since $a\notin B$, we have $a\notin\supp(\pi_C)$, so $f(a)=\pi_C(a)=a$.
  Letting $a\in A$ vary, we conclude that $f$ is the identity function
  off of $B$, while $f$ agrees with $\pi_C\in\Alt_B(A)$ on $B$.
  Hence, $f\in\Alt_B(A)$.
\end{proof}  

\bigskip

\noindent
6.B. {\bf Product Clones. }
Product clones were defined in Section~\ref{Section2} in the paragraph
preceding Corollary~\ref{cor-specprop3}.
Here we want to show that for large enough $\kappa$, both closure operators
$\ULC_\kappa$ and $\LC_\kappa$ commute with the
formation of product clones.
The special case $\kappa=\omega$ will be applied in the proof of
Theorem~\ref{thm-lots}(2).

\begin{lm}
  \label{lm-product}
  Let $\clo{P}$ be a clone on $A$ and
  $\clo{Q}$ a clone on $B$. 
  \begin{enumerate}
  \item[{\rm(i)}]
    $\ULC_\kappa(\clo{P}\times\clo{Q})=
    \ULC_\kappa(\clo{P})\times\ULC_\kappa(\clo{Q})$ for all $\kappa\ge4$, and
  \item[{\rm(ii)}]
    $\LC_\kappa(\clo{P}\times\clo{Q})=
    \LC_\kappa(\clo{P})\times\LC_\kappa(\clo{Q})$ for all $\kappa\ge4$.
  \end{enumerate}  
\end{lm}  

\begin{proof}
  Let $\kappa\ge4$.
  We know from Corollary~\ref{cor-specprop3} that both clones
  $\ULC_\kappa(\clo{P}\times\clo{Q})$ and $\LC_\kappa(\clo{P}\times\clo{Q})$
  are product clones on $A\times B$. Hence the equalities in
  statements (i)--(ii) will follow
  if we prove the following fact for all $0<n<\omega$ and all
  cardinals $\lambda<\kappa$:
  \begin{enumerate}
  \item[($\diamond_{n,\lambda}$)]
    a product operation $f\times g$,
    where $f$ is an $n$-ary operation on $A$ and $g$ is an $n$-ary operation on
    $B$, is $\lambda$-ultrainterpolable [$\lambda$-interpolable] by
    $\clo{P}\times\clo{Q}$ if and only if $f$ is
    $\lambda$-ultrainterpolable [$\lambda$-interpolable] by $\clo{P}$
    and $g$ is
    $\lambda$-ultrainterpolable [$\lambda$-interpolable] by $\clo{Q}$.
  \end{enumerate}

The part of ($\diamond_{n,\lambda}$) that refers to
$\lambda$-interpolability  
is an immediate consequence of the definitions.
Alternatively, one can use
Lemma~\ref{lm-invrel}(3) and the extension of
\cite[Satz~2.3.7(vi)]{pk}
to relations of arbitrary (possibly infinite) arity.
This proves the equality in Lemma~\ref{lm-product}(ii).

Now we prove the part of ($\diamond_{n,\lambda}$) that refers to
$\lambda$-ultrainterpolability. This will prove the
equality in Lemma~\ref{lm-product}(i).
Let $f\times g$ be a product operation as in ($\diamond_{n,\lambda}$). 
Recall from the definition that $f\times g$ is $\lambda$-ultrainterpolable by
$\clo{P}\times\clo{Q}$ if and only if 
\begin{enumerate}
  \item[(1)]
    the operation
    $(f\times g)_{\mathcal{U}}$ is $\lambda$-interpolable by the clone
    $(\clo{P}\times\clo{Q})_{\mathcal{U}}$
    for every ultrafilter $\mathcal{U}$ on any nonempty set $I$.
  \end{enumerate}
  Similarly, $f$ is $\lambda$-ultrainterpolable by $\clo{P}$
  and $g$ is $\lambda$-ultrainterpolable by $\clo{Q}$
  if and only if
  \begin{enumerate}
  \item[(2)]
  $f_{\mathcal{U}}$ is $\lambda$-interpolable by $\clo{P}_{\mathcal{U}}$
  and 
  $g_{\mathcal{U}}$ is $\lambda$-interpolable by $\clo{Q}_{\mathcal{U}}$
  for every ultrafilter $\mathcal{U}$ on any nonempty set $I$.
  \end{enumerate}
  The proof of ($\diamond_{n,\lambda}$) for $\lambda$-ultrainterpolability
  will be complete if we show that conditions (1) and (2) are equivalent.

  Fix $I$ and $\mathcal{U}$.
  The following map is a  bijection between the sets 
  $(A\times B)^I/\mathcal{U}$ and $(A^I/\mathcal{U})\times(B^I/\mathcal{U})$:
  \[
  \epsilon\colon
  (A\times B)^I/\mathcal{U}\to(A^I/\mathcal{U})\times(B^I/\mathcal{U}),
  \quad \bigl((a_i,b_i)\bigr)_{i\in I}/\mathcal{U}\mapsto
  \bigl((a_i)_{i\in I}/\mathcal{U},(b_i)_{i\in I}/\mathcal{U}\bigr).
  \]
  By identifying $(A\times B)^I/\mathcal{U}$ and 
  $(A^I/\mathcal{U})\times(B^I/\mathcal{U})$ via $\epsilon$ we see that
  for any two $n$-ary operations $h_A$ on $A$ and $h_B$ on $B$ we have that
  $(h_A\times h_B)_{\mathcal{U}}$ and $(h_A)_{\mathcal{U}}\times(h_B)_{\mathcal{U}}$
  are the same operation. In particular, this implies that
  $(f\times g)_{\mathcal{U}}$ and $f_{\mathcal{U}}\times g_{\mathcal{U}}$
  are the same operation, and
  $(\clo{P}\times\clo{Q})_{\mathcal{U}}$ and
  $\clo{P}_{\mathcal{U}}\times\clo{Q}_{\mathcal{U}}$ are the same clone.

  Thus, (1) is equivalent to the condition that
  $f_{\mathcal{U}}\times g_{\mathcal{U}}$ is $\lambda$-interpolable by
    $\clo{P}_{\mathcal{U}}\times\clo{Q}_{\mathcal{U}}$
  for every ultrafilter $\mathcal{U}$ on any nonempty set $I$.
  As we observed in the proof (ii), this condition is equivalent to
  condition~(2), which completes the proof.
\end{proof}  

\begin{cor}
  \label{cor-product}
Let $\clo{P}$ be a clone on $A$,
$\clo{Q}$ a clone on $B$, and let $\kappa\ge4$ be a cardinal.
The product clone $\clo{P}\times\clo{Q}$ is $\kappa$-ultraclosed
if and only if both $\clo{P}$ and $\clo{Q}$ are $\kappa$-ultraclosed.
\end{cor}

\bigskip

\noindent
6.C. {\bf Goldstern--Shelah clones. }
Given an infinite set $A$
and a maximal ideal $\mathcal{I}$ of the Boolean algebra
${\mathcal P}(A)$,
Goldstern and Shelah define in \cite[Definition~2.1]{goldstern-shelah}
a clone $\clo{C}(\mathcal{I})$ by
specifying that $f\in \clo{C}(\mathcal{I})$ iff for each $S\in \mathcal{I}$
we have $f(S,S,\ldots,S)\in \mathcal{I}$.
They prove that $\clo{C}(\mathcal{I})$
is a maximal clone on $A$, and that if $\mathcal{I}$ and $\mathcal{J}$ are distinct
maximal ideals of ${\mathcal P}(A)$,
then $\clo{C}(\mathcal{I})$ and $\clo{C}(\mathcal{J})$ are distinct
maximal clones on $A$.
It is known that there exist $2^{2^{|A|}}$-many
maximal ideals in ${\mathcal P}(A)$, so this construction
produces $2^{2^{|A|}}$-many maximal clones on $A$.
This number is the same as the number of all clones on $A$.

We now derive from Theorem~\ref{thm-nu} that
all the Goldstern--Shelah clones are $3$-ultraclosed, and hence are
ultralocally closed.

\begin{cor}
  \label{cor-G-Sh}
  Every Goldstern--Shelah clone $\clo{C}(\mathcal{I})$
  contains a ternary near unanimity
  operation. Consequently, every such clone
  is $3$-ultraclosed; that is, it satisfies
$\clo{C}(\mathcal{I})=\ULC_3(\clo{C}(\mathcal{I}))$.
\end{cor}

\begin{proof}
An operation $f\colon A^n\to A$ is called \emph{conservative}
if $f(a_1,\ldots,a_n)\in\{a_1,\ldots,a_n\}$ for every
tuple $(a_1,\ldots,a_n)\in A^n$.
If 
$f$ is a conservative operation on $A$,
$\mathcal{I}$ is a maximal ideal of ${\mathcal P}(A)$,
and $S\in \mathcal{I}$, 
then $f(S,\ldots,S)\subseteq S\in \mathcal{I}$,
so $f\in \clo{C}(\mathcal{I})$.
Since any set supports a conservative ternary near unanimity operation,
any Goldstern--Shelah clone $\clo{C}(\mathcal{I})$
contains a ternary near unanimity operation.
By Theorem~\ref{thm-nu} we have
$\clo{C}(\mathcal{I})=\ULC_3(\clo{C}(\mathcal{I}))$.
\end{proof}

\bigskip

\noindent
6.D. {\bf Proofs of Theorems~\ref{thm-lots}--\ref{thm-referee}.}
We start with Theorem~\ref{thm-lots}, which is about the
existence of large intervals in the clone lattice such that either
all clones in the interval are ultralocally closed, or none of them are
ultralocally closed.

\begin{proof}[Proof of Theorem~\ref{thm-lots}]
Let $A$ be an infinite set of cardinality $\nu$.  

To construct a large interval $[\clo{C}_1,\clo{D}_1]$ of ultralocally
closed clones on $A$ let $\clo{D}_1:=\clo{O}_A$.
By the discussion at the beginning of Subsection~6.C 
there are $2^{2^\nu}$
Goldstern--Shelah clones $\clo{C}(\mathcal{I})$ on $A$, each one a maximal 
subclone of $\clo{O}_A$, where
$\mathcal{I}$ runs over all maximal ideals of the Boolean algebra $\PP(A)$.
We also saw in the proof of Corollary~\ref{cor-G-Sh} that
if $f$ is a conservative ternary near unanimity operation on $A$
then $f\in\clo{C}(\mathcal{I})$ for all $\mathcal{I}$.
Therefore, if we let $\clo{C}_1:=\langle f\rangle$ for a fixed such $f$,
then we get that all
Goldstern--Shelah clones are in the interval $[\clo{C}_1,\clo{D}_1]$.
Every clone in this interval contains $f$, and hence is ultralocally closed
by Theorem~\ref{thm-nu}.
The Goldstern--Shelah clones in $[\clo{C}_1,\clo{D}_1]$
witness that the interval $[\clo{C}_1,\clo{D}_1]$ contains $2^{2^\nu}$ clones
that are maximal in $\clo{D}_1$, and hence
$[\clo{C}_1,\clo{D}_1]$ has size $2^{2^{\nu}}$.
This proves statement (1) of Theorem~\ref{thm-lots}
and all properties of the interval
$[\clo{C}_1,\clo{D}_1]$ claimed in the last paragraph of the theorem.

To show that there also exists a large interval of clones on $A$ such that
none of the clones in the interval are ultralocally closed,
we will first work in the lattice of clones on the set $A\times A$.
Consider the product clones 
$\clo{Alt}(A)\times\clo{C}$
where $\clo{Alt}(A)$ is the essentially unary clone generated by the
alternating group on $A$ (see Subsection~6.A), and
$\clo{C}$ is an arbitrary clone on $A$.
It follows from Claim~\ref{clm-product}(2)
that the product clones on $A\times A$ form an interval in the lattice of
clones on $A\times A$. So, the product clones $\clo{Alt}(A)\times\clo{C}$
with fixed first coordinate $\clo{Alt}(A)$ form a subinterval, namely
\begin{equation}
  \label{eq-interval}
  [\clo{Alt}(A)\times\langle\rangle,\clo{Alt}(A)\times\clo{O}_A],
\end{equation}  
where $\langle\rangle$ is the clone of projections on $A$ (i.e.,
the subclone of
$\clo{O}_A$ generated by the empty set of operations).
Clearly, the interval \eqref{eq-interval} is isomorphic to the lattice
of all clones on $A$; in particular, there are $2^{2^\nu}$ clones
in the interval that are maximal in the clone
$\clo{Alt}(A)\times\clo{O}_A$ at the top.
Moreover, Corollary~\ref{cor-product}
implies that none of the clones in the interval \eqref{eq-interval}
are ultralocally closed, because
by Theorem~\ref{thm-alt}, $\clo{Alt}(A)$ is not ultralocally closed.
  
Since $|A\times A|=|A|=\nu$,
the clone of all operations on $A\times A$ is isomorphic to the clone of
all operations on $A$.
Hence the result proved in the preceding
paragraph completes the proof 
of the existence of an interval $[\clo{C}_2,\clo{D}_2]$ with the
properties stated in part (2) and the last paragraph of 
Theorem~\ref{thm-lots}.
\end{proof}

\begin{remrk}
\label{remrk-referee}
  The proof of Theorem~\ref{thm-lots}(1) shows how to 
  find $2^{2^\nu}$ ultralocally closed, near unanimity clones
  on an infinite set $A$ of cardinality $\nu$.
  The referee of this paper suggested the following idea for constructing
  $2^{2^\nu}$ ultralocally closed, essentially unary clones on $A$.
  
  Assume $0,1$ are distinct elements of $A$, and let $\mathcal{A}$ be
  an \emph{independent family} of subsets of $A\setminus\{0,1\}$  
  (i.e., $\mathcal{A}$ is a free generating set of
  the Boolean subalgebra of \mbox{$\mathcal{P}(A\setminus\{0,1\})$}
  generated by $\mathcal{A}$).
  For each set $X\in\mathcal{A}$ let $f_X\colon A\to A$
  be the characteristic
  function of $X$ (i.e., $f_X(a)=1$ if $a\in X$ and $f_X(a)=0$ if
  $a\in A\setminus X$).
  The claim is that for every subset $\mathcal{S}$ of $\mathcal{A}$,
  \[
  f_X\in\ULC_\omega(\langle f_S:S\in\mathcal{S}\rangle)
  \quad\text{if and only if}\quad
  X\in\mathcal{S},
  \]
  hence the ultralocally closed clones
  $\ULC(\langle f_S:S\in\mathcal{S}\rangle)$
  ($\mathcal{S}\subseteq\mathcal{A}$), which are essentially unary by
  Corollary~\ref{cor-specprop1},
  form an ordered subset in the clone
  lattice on $A$ that is order isomorphic to the power set of $\mathcal{A}$.
  Since for every infinite set of size $\nu$ there exists an
  independent family $\mathcal{A}$ of subsets such that
  $|\mathcal{A}|=2^\nu$, this construction yields an ordered set
  of essentially unary, ultralocally closed
  clones on $A$ that is order isomorphic to the power set
  of a $2^\nu$-element set.
\end{remrk}

Our second result to be proved here is Theorem~\ref{thm-omega1-alg},
which is about the algebraicity degree of $\ULC_{\omega}\langle-\rangle$.

\begin{proof}[Proof of Theorem~\ref{thm-omega1-alg}]
Let $A$ be any infinite set.  
For the proof of statement (1),
which asserts that the closure operator $\ULC_\omega\langle-\rangle$
is not algebraic,
we will use the clones $\clo{Alt}(A)$ and
$\clo{Alt}_B(A)$ discussed in Subsection~6.A. 
It is clear from the definition of $\clo{Alt}(A)$ that every
finite subset of $\clo{Alt}(A)$ is contained in $\clo{Alt}_B(A)$ for some
finite $B\subseteq A$. We also know from Lemma~\ref{lm-altB} that
each such clone $\clo{Alt}_B(A)$ is ultralocally closed.
Therefore
\begin{align*}
\bigcup\{\ULC_\omega(\langle F\rangle):F\subseteq\clo{Alt}(A),\,|F|<\omega\}
&{}\subseteq
\bigcup\{\ULC_\omega(\clo{Alt}_B(A)):B\subseteq A,\,|B|<\omega\}\\
&{}=
\bigcup\{\clo{Alt}_B(A):B\subseteq A,\,|B|<\omega\}\\
&{}=\clo{Alt}(A).
\end{align*}
Actually, $=$ holds in place of $\subseteq$ above, because
every term $\ULC_\omega(\clo{Alt}_B(A))$ ($|B|<\omega$)
in the union on the right hand side
can be rewritten as $\ULC_\omega(\langle\Alt_B(A)\rangle)$, where
$\Alt_B(A)$ is a finite set of permutations of $A$. Hence, every term 
in the union on the right hand of side of $\subseteq$ 
appears as a term in the union on the left hand side
as well, proving that $\supseteq$ also holds.
This implies that
\[
\clo{Alt}(A) =
\bigcup\{\ULC_\omega(\langle F\rangle):F\subseteq\clo{Alt}(A),\,|F|<\omega\}.
\]
On the other hand, we have by Theorem~\ref{thm-alt} that
\[
\clo{Alt}(A)\subsetneq\clo{Sym}_\omega(A)=\ULC_\omega(\clo{Alt}(A)).
\]
This proves that the closure operator $\ULC_\omega\langle-\rangle$
is not algebraic.

For claim~(2), 
which states that the closure operator $\ULC_\omega\langle-\rangle$
is $\omega_1$-algebraic,
it suffices to show that the following equality holds for any set
$G$ of operations on $A$:
\begin{equation}
  \label{eq-omega1-alg}
  \ULC_\omega(\langle G\rangle)=
  \bigcup\{\ULC_\omega(\langle F\rangle):F\subseteq G,\,|F|\le\omega\}.
\end{equation}
Indeed, \eqref{eq-omega1-alg} immediately implies that for any
set $G$ of operations on $A$,
\[
  G=\ULC_\omega(\langle G\rangle)\ \ \Leftrightarrow\ \ 
  G=\bigcup\{\ULC_\omega(\langle F\rangle):F\subseteq G,\,|F|\le\omega\},
\]
which is the defining property 
for $\ULC_\omega\langle-\rangle$ to be
$\omega_1$-algebraic. (See \eqref{eq-def-kappa-alg}.)

Now we prove \eqref{eq-omega1-alg}.
The inclusion $\supseteq$ 
holds because $\ULC_{\omega}\langle-\rangle$
is a closure operator.
For the reverse inclusion,
let $f$ be an operation in $\ULC_\omega(\langle G\rangle)$, say $f$ is $n$-ary.
By Corollary~\ref{cor-char-kappa}, this means that
\begin{enumerate}
  \item[$(\ddagger)$]
for every $k<\omega$,
    $A^n\,\bigl(=\dom(f)\bigr)$ has a finite cover
    $\CCf_k\,\bigl(\subseteq \PP(A^n)\bigr)$ such that
    whenever $\mathcal{B}\subseteq\CCf_k$
    satisfies $|\mathcal{B}|\le k$, there exists an $n$-ary
    $t^{[\mathcal{B}]}\in\langle G\rangle$ 
    such that $f|_{\bigcup{\mathcal{B}}}=t^{[\mathcal{B}]}|_{\bigcup{\mathcal{B}}}$.
\end{enumerate}
For each fixed $k<\omega$,
there are finitely many choices for $\mathcal{B}$, and
for each choice of $\mathcal{B}$,
the operation $t^{[\mathcal{B}]}\in\langle G\rangle$
is generated by a finite subset of $G$.
Therefore there exists a finite subset $F_k$ of $G$ such that
condition $(\ddagger)$ holds for that $k$ with $F_k$ in place of $G$.
Hence, by letting $F:=\bigcup\{F_k:k<\omega\}$,
we see that $|F|\le\omega$, and $(\ddagger)$ holds for $F$ in place of $G$.
This shows that $f\in\ULC_\omega(\langle F\rangle)$, and completes the proof
of \eqref{eq-omega1-alg} and statement (2).

The last statement of Theorem~\ref{thm-omega1-alg} is a reformulation of the
statement that the closure operator $\ULC_\omega\langle-\rangle$ is
$\omega_1$-algebraic.
\end{proof}

Next we prove Theorem~\ref{thm-not-kappa-alg} 
about the algebraicity degree of $\LC_{\omega}\langle-\rangle$.

\begin{proof}[Proof of Theorem~\ref{thm-not-kappa-alg}]
Let $A$ be an infinite set of cardinality $\nu$.
Our task is to show that
the closure operator $\LC_\omega\langle-\rangle$ on $\clo{O}_A$
is not $\kappa$-algebraic for any infinite regular cardinal $\kappa\le\nu$.
For each subset $B$ of $A$ let $\Inj_B(A)$ denote the set of all injective
functions $f\colon A\to A$ with `support' in $B$, by which we mean
all injective functions $f\colon A\to A$
satisfying $f(a)=a$ for all $a\in A\setminus B$.
It is clear that, for each subset $B\subseteq A$,
$\Inj_B(A)$ is closed under composition, and hence
it generates an essentially unary clone $\clo{Inj}_B(A)$ with unary part
$\Inj_B(A)$. We claim that the clone $\clo{Inj}_B(A)$ is locally closed.
Indeed, by Corollary~\ref{cor-specprop1},
$\LC_\omega\bigl(\clo{Inj}_B(A)\bigr)$ is an essentially unary clone, and
every unary operation in it is injective. Furthermore, since every injective
function $g\colon A\to A$ in $\LC_\omega\bigl(\clo{Inj}_B(A)\bigr)$ agrees, on 
each singleton set $\{a\}\subseteq A\setminus B$,
with some function in $\clo{Inj}_B(A)$, we get that $g\in\Inj_B(A)$.
This implies that $\clo{Inj}_B(A)$ is a
locally closed clone for every set $B\subseteq A$.

We will use these clones
to show that $\LC_\omega\langle-\rangle$ is not
$\kappa$-algebraic for any infinite regular $\kappa\le\nu$.
Fix such a $\kappa\le\nu=|A|$.
In what follows, a set $X$ is called \emph{$\kappa$-small} if $|X|<\kappa$.
Let 
\[
\clo{G}:=\bigcup\{\clo{Inj}_B(A):B\subseteq A, |B|<\kappa\}
\]
(which is a union of clones), 
and let $G:=\bigcup\{\Inj_B(A):B\subseteq A, |B|<\kappa\}$
(which is a union of sets of unary functions).
$G$ is closed under composition, for if $f_i\in\Inj_{B_i}(A)$ with
$|B_i|<\kappa$ ($i<2$), then $f_1\circ f_0\in\Inj_{B_0\cup B_1}(A)$, and 
$|B_0\cup B_1|<\kappa$.
It follows that $\clo{G}$
is an essentially unary clone with unary part $G$.
Our goal is to show that
\begin{equation}
  \label{eq1-not-kappa-alg}
\clo{G}=\bigcup\{\LC_\omega(\langle F\rangle):F\subseteq\clo{G},\,|F|<\kappa\}
\end{equation}
and
\begin{equation}
  \label{eq2-not-kappa-alg}
  \clo{G}\subsetneq\LC_\omega(\clo{G})=\LC_\omega(\langle\clo{G}\rangle).
\end{equation}
This will prove that the closure operator 
$\LC_\omega\langle-\rangle$ on $\clo{O}_A$ is not $\kappa$-algebraic.
(See \eqref{eq-def-kappa-alg}.)

In \eqref{eq1-not-kappa-alg} the inclusion $\subseteq$ clearly holds,
because
$f\in\clo{G}$ implies that $f\in\LC_\omega(\langle F\rangle)$ for
$F=\{f\}\subseteq\clo{G}$ with $|F|=1<\kappa$.
To prove the reverse inclusion, recall that
each operation $f\in\clo{G}$ is a member of $\clo{Inj}_B(A)$ for
some $\kappa$-small subset $B\subseteq A$. Therefore 
for every
$\kappa$-small
set $F\subseteq\clo{G}$
which appears on the right hand side of
\eqref{eq1-not-kappa-alg}
there exists a `support selecting function'
$f\mapsto B_f$ such that
$f\in\clo{Inj}_{B_f}(A)$ and $|B_f|<\kappa$ for all $f\in F$.
Since $\clo{Inj}_B(A)\subseteq\clo{Inj}_{B'}(A)$ whenever
$B\subseteq B'\,(\subseteq A)$, we see that
$F\subseteq\clo{Inj}_{B_F}(A)$
holds for the set $B_F:=\bigcup\{B_f:f\in F\}$.
Since $F$ is $\kappa$-small, each $B_f$ is $\kappa$-small,
and $\kappa$ is regular, $B_F$ is also $\kappa$-small.
Thus, $\langle F\rangle\subseteq\clo{Inj}_{B_F}(A)$. We proved
earlier that the clone $\clo{Inj}_{B_F}(A)$ is locally closed,
therefore we obtain that
$\LC_\omega(\langle F\rangle)\subseteq\LC_\omega(\clo{Inj}_{B_F}(A))
=\clo{Inj}_{B_F}(A)$. 
This inclusion holds for every
$\kappa$-small set $F\subseteq\clo{G}$ 
on the right hand side of \eqref{eq1-not-kappa-alg}, and so does the 
inequality $|B_F|<\kappa$. Hence,
the right hand side of \eqref{eq1-not-kappa-alg} is contained in $\clo{G}$.

In \eqref{eq2-not-kappa-alg} the equality $=$ holds,
because $\clo{G}$ is a clone,
and hence $\clo{G}=\langle\clo{G}\rangle$.
For the inclusion $\subsetneq$ 
recall that $\kappa$ is an infinite cardinal such that
$\kappa\le\nu=|A|$. Furthermore,
by its definition, $\clo{G}$ is an essentially
unary clone whose unary part $G$ 
consists of all injective functions $A\to A$
of $\kappa$-small support.
Therefore, $G$ does not contain all injections $A\to A$.
By Corollary~\ref{cor-specprop1}, the clone
$\LC_\omega(\clo{G})$ is also essentially unary, and its unary part
consists of injections $A\to A$. However, the unary part of
$\LC_\omega(\clo{G})$ does contain all injections $A\to A$, because
every injective function $A\to A$ is interpolable, on each finite set
$S\subseteq A$, by injections of $\kappa$-small support.
\end{proof}  

Finally, we prove Theorem~\ref{thm-referee}.

\begin{proof}[Proof of Theorem~\ref{thm-referee}]
Let $A$ be an infinite set of cardinality $\nu$, and let $F\subseteq\clo{O}_A$.
The main statement of the theorem is that if the clone
$\ULC_\omega(\langle F\rangle)$
is uncountable and contains a near unanimity operation, then
$|F|=|\ULC_\omega(\langle F\rangle)|$.
The second statement concerns the special case when 
$\ULC_\omega(\langle F\rangle)=\clo{O}_A$. In this special case
$\ULC_\omega(\langle F\rangle)$ clearly contains a near unanimity operation
and 
$|\ULC_\omega(\langle F\rangle)|=|\clo{O}_A|=2^\nu$, so
$\ULC_\omega(\langle F\rangle)$ is uncountable. Therefore the main
statement yields the desired conclusion that
$|F|=|\ULC_\omega(\langle F\rangle)|=2^\nu$.

To prove the main statement, assume 
$\clo{C}:=\ULC_\omega(\langle F\rangle)$ is uncountable and $h$ is a
near unanimity operation in $\clo{C}$. In particular, we have that
$\clo{C}=\ULC_\omega(\clo{C})$
and 
$\langle F\cup\{h\}\rangle\subseteq\clo{C}$. Thus,
\[
\clo{C}=\ULC_\omega(\langle F\rangle)
\subseteq\ULC_\omega(\langle F\cup\{h\}\rangle)
\subseteq\ULC_\omega(\clo{C})=\clo{C},
\]
which implies that $\clo{C}=\ULC_\omega(\langle F\cup\{h\}\rangle)$.
However, since the clone $\langle F\cup\{h\}\rangle$ contains a
near unanimity operation, we know from Theorem~\ref{thm-nu} that
it is ultralocally closed, that is
$\ULC_\omega(\langle F\cup\{h\}\rangle)=\langle F\cup\{h\}\rangle$.
Hence, $\clo{C}=\langle F\cup\{h\}\rangle$, and therefore
$|\clo{C}|=|F|+\aleph_0$.
Now the assumption that $\clo{C}=\ULC_\omega(\langle F\rangle)$
is uncountable implies that $|F|=|\clo{C}|=|\ULC_\omega(\langle F\rangle)|$,
as claimed.
\end{proof}

\section*{Acknowledgments}
We thank Miguel Campercholi and Diego Casta{\~{n}}o
for correcting a mistake in an earlier version of this
article. We also thank the referee for many helpful comments, including
the question that led to Theorem~\ref{thm-referee} and the suggestions on
constructing large families of ultralocally closed clones, one of which is
reproduced in Remark~\ref{remrk-referee}.

\bibliographystyle{plain}

\end{document}